\documentclass[a4paper,10pt]{amsart}
\usepackage[utf8]{inputenc}
\usepackage{amsmath, amsthm, graphicx, wrapfig, wasysym, fullpage, amsfonts, amssymb, array, color}

\newtheorem{Thm}{Theorem}
\newtheorem{Prop}{Proposition}[section]
\newtheorem{Lemma}[Prop]{Lemma}
\newtheorem{Cor}[Prop]{Corollary}
\newtheorem{Def}[Prop]{Definition}
\newtheorem*{Def*}{Definition}
\theoremstyle{remark}

\numberwithin{equation}{section}

\newcommand{\Wie}{\operatorname{W}}
\newcommand{\WL}{\operatorname{WL}}
\newcommand{\WR}{\operatorname{WR}}

\title[TFPLs of Excess 2]{Triangular Fully Packed Loop Configurations of Excess 2}
\author{Sabine Beil}
\address{Sabine Beil, Universit\"at Wien, Fakult\"at f\"ur Mathematik, Oskar-Morgenstern-Platz~1, 1090 Wien, Austria}
\email{sabine.beil@univie.ac.at}
\thanks{Supported by the Austrian Science Foundation FWF, START grant Y463.}

\begin{document}
\maketitle

\begin{abstract}
Triangular fully packed loop configurations (TFPLs) came up in the study of fully packed loop configurations on a square (FPLs) corresponding to 
link patterns with a large number of nested arches. To a TFPL is assigned a triple $(u,v;w)$ of $01$-words encoding its boundary conditions which must necessarily satisfy that $d(u)+d(v)\leq d(w)$, where $d(u)$ denotes the number of inversions in $u$. 
Wieland gyration, on the other hand, was invented to show the rotational invariance of the numbers $A_\pi$ of FPLs corresponding to a given link pattern $\pi$. 
Later, Wieland drift -- a map on TFPLs that is based on Wieland gyration -- was defined. 
The main contribution of this article is a linear expression for the number of TFPLs with boundary $(u,v;w)$ where \hbox{$d(w)-d(u)-d(v)=2$} in terms of numbers of stable TFPLs, that is, TFPLs invariant under Wieland drift.
This linear expression is consistent with already existing enumeration results for TFPLs with boundary $(u,v;w)$ where \hbox{$d(w)-d(u)-d(v)=0,1$}.
\end{abstract}

\section{Introduction}

The basis for this article is the \textit{fully packed loop model} that has its origin in the \textit{six-vertex} model (which is also called \textit{square ice} model) of statistical mechanics; 
a fully packed loop configuration (FPL) of size $n$ is a subgraph $F$ of the $n\times n$-square grid together with $2n$ external edges such that each of the $n^2$ vertices is of 
degree~$2$ in $F$ and every other external edge is occupied by $F$ starting with the topmost horizontal external edge on the left side. See Figure~\ref{Fig:FPL} for an example. 
\begin{figure}[tbh]
 \includegraphics[width=.6\textwidth]{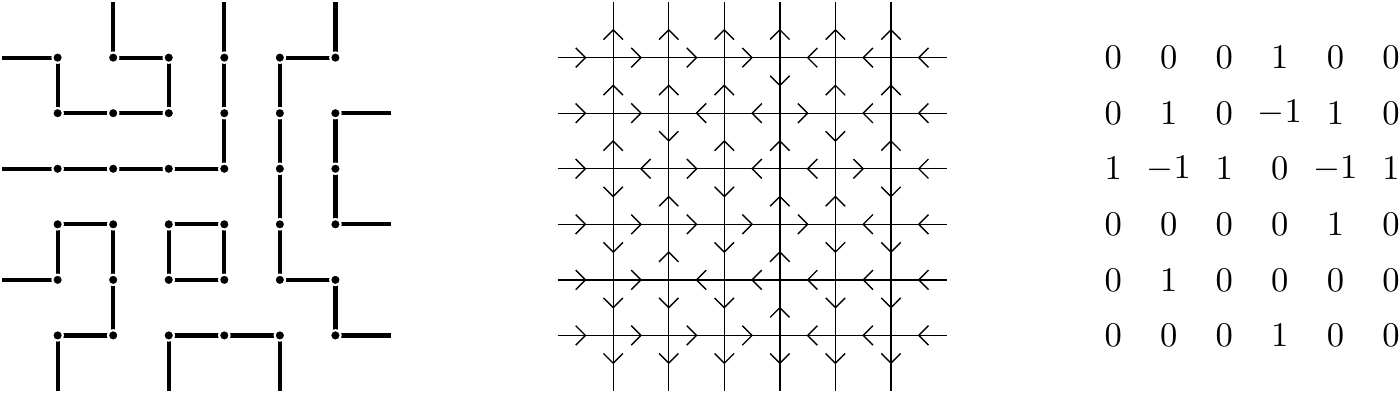}
 \caption{An FPL, a six vertex configuration and an ASM.}
 \label{Fig:FPL}
\end{figure}
FPLs are significant to algebraic combinatorics due to their one-to-one correspondence to \textit{alternating sign matrices} (ASMs). This is why FPLs of size $n$ are enumerated by the famous formula for the number of 
ASMs of size $n$ proved in~\cite{Zeilberger}. 

In contrast to alternating sign matrices, FPLs allow a refined study in dependency on the connectivity of the occupied external edges 
(these connections are encoded as a \textit{link pattern}). The study of FPLs having a link pattern with \textit{nested arches} is an example of one such refined study; it was conjectured in~\cite{Zuber} and later 
proved in~\cite{CKLN} that the number of FPLs having a fixed link pattern $\pi\cup m$ consisting of a link pattern $\pi$ of size $n$ and $m$ nested arches is polynomial in $m$. In the course of the proof of this 
conjecture \textit{triangular fully packed loop configurations} (TFPLs) came up. To be more precise, the following expression for the number $A_\pi(m)$ of FPLs having link pattern $\pi\cup m$ 
including numbers $t_{u,v}^w$ of TFPLs satisfying certain boundary conditions encoded by a triple $(u,v;w)$ of $01$-words was shown:
\begin{equation}
A_\pi(m)=\sum\limits_{u,v\in\mathcal{D}_n} P_{\lambda(u)}(n) t_{u',v'}^{\textbf{w}(\pi)} P_{\lambda(v)'}(m-2n+1),
\label{Eq:FPL_nested_arches}\end{equation}
where the sum runs over all Dyck words $u,v$ of length $2n$, $u'$ denotes the $01$ word obtained from a Dyck word $u$ by deleting the first $0$ and the last $1$, $\textbf{w}(\pi)$ denotes the Dyck word corresponding to 
the link pattern $\pi$, $\lambda(u)$ denotes the Young diagram associated with a $01$ word $u$, $\lambda'$ denotes the conjugate of a Young diagram $\lambda$ and 
\begin{equation}P_\lambda(x)=\prod\limits_{\mathfrak{C}\in\lambda}\frac{x+c(\mathfrak{C})}{h(\mathfrak{C})}
\notag\end{equation}
with $c(\mathfrak{C})$ being the content of the cell $\mathfrak{C}$ and $h(\mathfrak{C})$ the hook length of $\mathfrak{C}$. 
Apperantly, Equation~(\ref{Eq:FPL_nested_arches}) motivates the study of TFPLs and the numbers $t_{u,v}^w$. Another motivation for their study comes from the many nice properties of TFPLs which have been discovered 
since the emergence of TFPLs, see~\cite{Thapper},~\cite{Nadeau1} and~\cite{TFPL}. An example of one such property is that the boundary $(u,v;w)$ of a TFPL has to fulfill that $d(u)+d(v)\leq d(w)$, 
where $d(\omega)$ denotes the number of inversions in a word $\omega$; the integer 
\begin{equation}
\operatorname{exc}(u,v;w)=d(w)-d(u)-d(v)
\notag\end{equation} 
is said to be the \textit{excess} of $u,v,w$. To study TFPLs with respect to the excess of their boundary turned out to be fruitful; in~\cite{TFPL} enumeration results for TFPLs with boundary $(u,v;w)$ where $exc(u,v;w)=0,1$ 
were proved.

\begin{figure}[tbh]
\includegraphics[width=.4\textwidth]{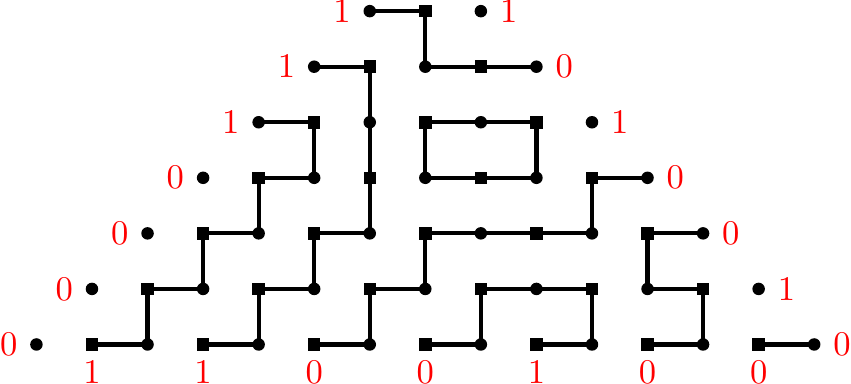}
 \caption{A triangular fully packed loop configuration with boundary $(0000111,1010010;1100100)$.}
\end{figure}

\textit{Wieland gyration}, on the other hand, is an operation on FPLs that was invented in~\cite{Wieland1} to prove the rotational invariance of the numbers $A_\pi$ of FPLs corresponding to given link patterns $\pi$. 
Later it was heavily used by Cantini and Sportiello~\cite{CanSport} to prove the Razumov--Stroganov conjecture. In connection with TFPLs, Wieland gyration first appeared in~\cite{Nadeau1} 
following work of~\cite{Thapper}, after which \textit{Wieland drift} was introduced in~\cite{WielandDrift} as the natural definition of Wieland gyration for TFPLs. In contrast to Wieland gyration, Wieland drift 
is not an involution. It was shown in~\cite{WielandDrift} that Wieland drift is eventually periodic with period~$1$.\\

This article will focus on TFPLs with boundary $(u,v;w)$ where $exc(u,v;w)=2$. The main contribution of this paper will be a linear expression for $t_{u,v}^w$ in terms of numbers of \textit{stable} TFPLs, 
that is, TFPLs invariant under the application of Wieland drift. This linear expression is consistent with the already existing enumeration results for TFPLs with boundary $(u,v;w)$ where $exc(u,v;w)=0,1$. 

To give the exact formulation of the main result of this article further notation is needed: let $\sigma$ and $\tau$ be two $01$-words of length $N$. Then  
\begin{itemize}
 \item[--] $\vert \omega\vert_i$ will denote the number of occurrences of $i$ in $\sigma$;
 \item[--] it will be written $\sigma\leq \tau$ if $\vert \sigma_1\cdots\sigma_m\vert_1\leq \vert \tau_1\cdots\tau_m\vert_1$ for all $1\leq m\leq N$;
\item[--] $\overleftarrow{\sigma}$ will denote the word $\sigma_N\cdots\sigma_1$, $\overline{\sigma}$ will denote the word
$\overline{\sigma_1}\cdots\overline{\sigma_n}$ where $\overline{0}=1$ and $\overline{1}=0$ and $\sigma^\ast$ will denote the word $\overleftarrow{\overline{\sigma}}$.
 \end{itemize}

\begin{Def*}
Let $\sigma$ and $\sigma^+$ be two $01$-words of the same length wich satisfy $\sigma\leq\sigma^+$ and $d(\sigma^+)-d(\sigma)\leq 2$. Then set 
\begin{equation}
g_{\sigma,\sigma^+}= \begin{cases} 1 & ,if\, \sigma^+=\sigma;\\
\vert\sigma_R\vert_1+1 & ,if\,\, \sigma=\sigma_L01\sigma_R \,\, and\,\,\sigma^+=\sigma_L10\sigma_R;\\
(\vert\sigma_R\vert_1+1)(\vert\sigma_R\vert_1+\vert\sigma_M\vert_1+1) & ,if\,\,\sigma=\sigma_L01\sigma_M01\sigma_R \,\, and \,\, \sigma^+=\sigma_L10\sigma_M10\sigma_R;\\
\frac{(\vert\sigma_R\vert_1+1)\vert\sigma_R\vert_1}{2} & ,if\,\,\sigma=\sigma_L001\sigma_R \,\, and\,\,\sigma^+=\sigma_L100\sigma_R;\\
\frac{(\vert\sigma_R\vert_1+2)(\vert\sigma_R\vert_1+1)}{2} & ,if\,\,\sigma=\sigma_L011\sigma_R \,\, and\,\,\sigma^+=\sigma_L110\sigma_R,\\
\end{cases}
\notag\end{equation}
where in each case $\sigma_L$, $\sigma_M$ and $\sigma_R$ are appropriate $01$-words.
\end{Def*}

In the following, denote by $s_{u^+,v^+}^w$ the number of stable TFPLs with boundary $(u^+,v^+;w)$. 

\begin{Thm}\label{Thm:Expressing_in_stable_TFPLs_exc_2} Let $u,v,w$ be words of the same length such that $d(w)-d(u)-d(v)=2$. Then 
\begin{equation}
t_{u,v}^w=\sum\limits_{\substack{u^+\geq u,\, v^+\geq v:\\ \operatorname{exc}(u^+,v^+;w)\geq 0}} g_{u,u^+}\ g_{v^\ast,(v^+)^\ast}\ s_{u^+,v^+}^w.
\label{Eq:Thm_Expressing_in_stable_TFPLs_exc_2}
\end{equation} 
\end{Thm}
For example, 
\begin{align}
t_{0011,0110}^{1100}=3=\sum\limits_{\substack{u^+\geq 0011,\, v^+\geq 0110:\\ \operatorname{exc}(u^+,v^+;w)\geq 0}} g_{0011,u^+}\ g_{1001,(v^+)^\ast}\ s_{u^+,v^+}^{1100}.
\notag\end{align}

By a result in~\cite{TFPL} it holds 
\begin{equation}
t_{u,v}^w=s_{u,v}^w,
\label{Eq:Thm_Expressing_in_stable_TFPLs_exc_0}
\end{equation}
if $\operatorname{exc}(u,v;w)=0$. For that reason, the linear expression for the number of TFPLs with boundary $(u,v;w)$ where $\operatorname{exc}(u,v;w)=1$ in terms of TFPLs of excess $0$ 
proved in~Theorem~6.16(5) in~\cite{TFPL} can be written as follows:

\begin{equation}
t_{u,v}^w=s_{u,v}^w+\sum\limits_{\substack{u^+> u:\\ \operatorname{exc}(u^+,v;w)=0}}g_{u,u^+} \ s_{u^+,v}^{w} + \sum\limits_{\substack{v^+> v:\\ \operatorname{exc}(u,v^+;w)=0}}g_{v^\ast,(v^\ast)^+} \ s_{u,v^+}^{w}. 
\label{Eq:Thm_Expressing_in_stable_TFPLs_exc_1}
\end{equation}

Summing up, the linear expression stated in Theorem~\ref{Thm:Expressing_in_stable_TFPLs_exc_2} is consistent with the already existing enumeration results for TFPLs with boundary $(u,v;w)$ where 
$\operatorname{exc}(u,v;w)=0,1$. This suggests a study of TFPLs with boundary $(u,v;w)$ where $\operatorname{exc}(u,v;w)\geq 3$ based on the methods presented in this article in order to obtain expressions for the 
numbers $t_{u,v}^w$ in terms of stable TFPLs.\\

\textcolor{blue}{A poster about this work will be presented at FPSAC 2015.}

\section{Preliminaries}\label{Sec:Preliminaries}

\subsection{Words and Young diagrams}
A \textit{word} $\omega$ of length $N$ is a finite sequence $\omega=\omega_1\omega_2\cdots \omega_N$ where $\omega_i\in \{0,1\}$ for all 
$1\leq i\leq N$. Given a word $\omega$ the number of occurrences of 0 (\textit{resp.} 1) in $\omega$ is denoted by $\vert \omega\vert_0$ (\textit{resp.} $\vert \omega\vert_1$). 
Furthermore, it is said that two words $\omega, \sigma$ of length $N$ with the same number of occurrences of 1 satisfy $\omega\leq\sigma$ if 
$\vert\omega_1\cdots\omega_n\vert_1\leq\vert\sigma_1\cdots\sigma_n\vert_1$ holds for all $1\leq n\leq N$. Finally, the number of inversions of $\omega$ that is pairs
$1\leq i<j\leq N$ satisfying $\omega_i=1$ and $\omega_j=0$ is denoted by $d(\omega)$.

Throughout this article, in a Young diagram empty columns and empty rows are allowed. With a word $\omega$ a Young diagram $\lambda(\omega)$ will be associated as follows: 
to a given word $\omega$ a path on the square lattice is constructed by drawing a $(0,1)$-step if $\omega_i=0$ and a $(1,0)$-step if $\omega_i=1$ for i from $1$ to $n$. Additionally, 
a vertical line through the path's starting point and a horizontal line through its ending point are drawn. Then the region enclosed by the lattice path and the two lines is a Young diagram 
which shall be the image of $\omega$ under $\lambda$. In Figure~\ref{Fig:BijectionWordsYoung}, an example of a word and its corresponding Young diagram is given. 
For two words $\omega$ and $\sigma$ of length $N$ it then holds $\omega\leq \sigma$ if and only if $\lambda(\omega)$ is contained in $\lambda(\sigma)$. 
Furthermore, the number of cells of $\lambda(\omega)$ equals $d(\omega)$.

\begin{figure}[tbh]
\centering
\includegraphics[width=.35\textwidth]{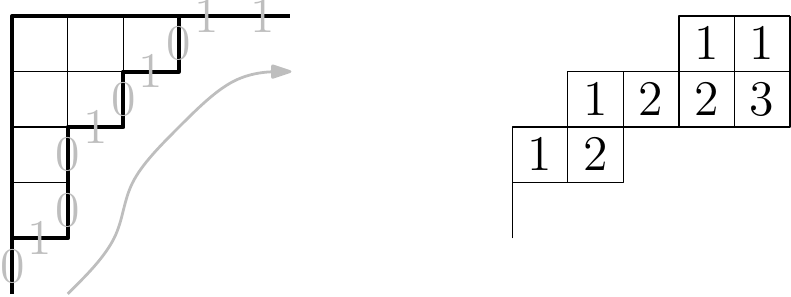}
\caption{The Young diagram $\lambda(0100101011)$ and a semi-standard Young tableau of skew shape $\lambda(011011100)/\lambda(001011011)$.}
\label{Fig:BijectionWordsYoung}
\end{figure}

There are skew shaped Young diagrams which play an important role in the context of Wieland drift: a skew shape is said to be a \emph{horizontal strip} (\emph{\textit{resp.} a vertical strip}) 
if each of its columns (\emph{\textit{resp.}} rows) contains at most one cell. 
Consider two words $\omega$ and $\sigma$ satisfying $\vert \omega\vert_1=\vert \sigma\vert_1$, $\vert\omega\vert_0=\vert \sigma\vert_0$ and $\omega\leq\sigma$. 
Then the skew shape \hbox{$\lambda(\sigma)/\lambda(\omega)$} is a horizontal strip (\emph{\textit{resp.}} a vertical strip) if and only if for each 
\hbox{$j\in\{1,\dots,\vert \omega\vert_1\}$} (\textit{resp.} for each \hbox{$j\in\{1,2,\dots,\vert\omega\vert_0\}$}) the following holds: 
{\it If $\omega_i$ is the \hbox{$j$-th} one (\emph{\textit{resp.}} zero) in $\omega$ then $\sigma_{i-1}$ or $\sigma_i$ (\emph{\textit{resp.}} $\sigma_{i}$ or $\sigma_{i+1}$) is the \hbox{$j$-th} one 
(\emph{\textit{resp.}} zero) in $\sigma$.}
In the following, if the skew shaped Young diagram \hbox{$\lambda(\sigma)/\lambda(\omega)$} is a horizontal strip (\textit{resp.} a vertical strip) it will be written 
\hbox{$\omega\stackrel{\mathrm{h}}{\longrightarrow}\sigma$} (\textit{resp.} \hbox{$\omega\stackrel{\mathrm{v}}{\longrightarrow}\sigma$}).

Semi-standard Young tableaux of skew shape 
$\lambda(\sigma)/\lambda(\omega)$ with entries $1,2,\dots,m$ are in bijection with sequences of Young diagrams  
\begin{equation}
\lambda(\omega)=\lambda(\tau^{0})\subseteq\lambda(\tau^{1})\subseteq\cdots\subseteq\lambda(\tau^{m-1})\subseteq\lambda(\tau^{m})=\lambda(\sigma),  
\notag\end{equation}
such that \hbox{$\tau^{i-1}\stackrel{\mathrm{h}}{\longrightarrow}\tau^{i}$} for each $1\leq i\leq m$. 
To be more precise, the horizontal strip $\lambda(\tau^{i})/\lambda(\tau^{i-1})$ gives the cells of the semi-standard Young tableau of skew-shape $\lambda(\sigma)/\lambda(\omega)$ that have entry $i$ for 
$1\leq i\leq m$. 
For instance, the semi-standard Young tableau of skew shape $\lambda(011011100)/\lambda(001011011)$ in Figure~\ref{Fig:BijectionWordsYoung} corresponds to the sequence 
\begin{equation}
 \lambda(001011011)\subseteq\lambda(010101110)\subseteq\lambda(011011010)\subseteq\lambda(011011100).
\notag\end{equation}

\subsection{Triangular fully packed loop configurations}\label{Subsec:TFPL_exc_2}
To give the definition of triangular fully packed loop configurations the following graph is needed:

\begin{Def}[The graph $G^N$] Let $N$ be a positive integer. The graph $G^N$ is defined as the induced subgraph of the square grid made up of $N$ consecutive centered rows of 
\hbox{$3,5,\dots, 2N+1$} vertices from top to bottom together with $2N+1$ vertical \emph{external} edges incident to the $2N+1$ bottom vertices. 
\end{Def}

\begin{figure}[tbh]
\centering
\includegraphics[width=.5\textwidth]{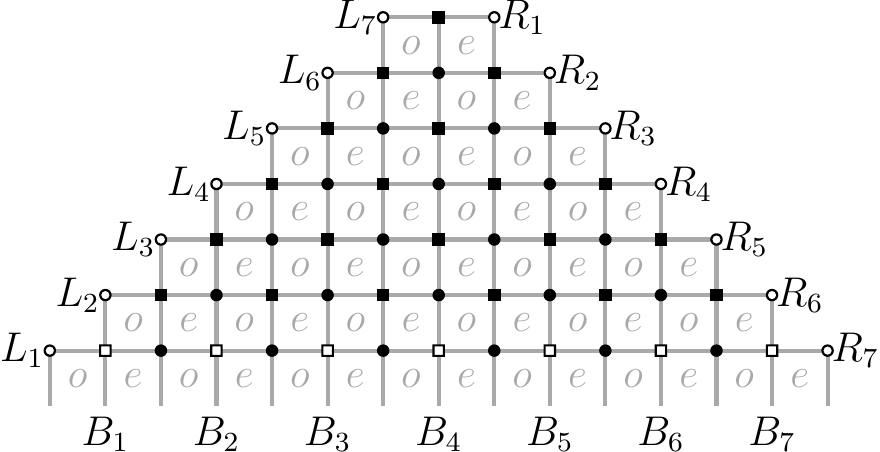}
\caption{The graph $G^7$.}
\label{Fig003}
\end{figure}

In Figure~\ref{Fig003}, the graph $G^7$ is depicted. 
From now on, the vertices of $G^N$ are partitioned into \textit{odd} and \textit{even} vertices in a chessboard manner where by convention the leftmost vertex of the top row of $G^N$ is odd.
In the figures, odd vertices are represented by circles and even vertices by squares. There are vertices of $G^N$ that play a special role:
let \hbox{$\mathcal{L}^N=\{L_1,L_2,\dots,L_N\}$} (\textit{resp.} \hbox{$\mathcal{R}^N=\{R_1,R_2,\dots,R_N\}$}) be the set made up of the vertices which are leftmost (\textit{resp.} rightmost) in each of the $N$ rows of $G^N$
and let $\mathcal{B}^N=\{B_1,B_2,\dots,B_N\}$ be the set made up of the even vertices of the bottom row of $G^N$. The vertices are numbered from left to right. 
Furthermore, the $N(N+1)$ unit squares of $G^N$ including external unit squares that have three surrounding
edges only are said to be the \emph{cells} of $G^N$. They are partitioned into \emph{odd} and \emph{even} cells in a chessboard manner where by convention the top left 
cell of $G^N$ is odd. 

\begin{Def}[Triangular fully packed loop configuration] \label{Def:TFPL}
Let $N$ be a positive integer.
A \emph{triangular fully packed loop configuration} (TFPL) of size $N$ is a subgraph $f$ of $G^N$ such that: 
\begin{enumerate}
 \item Precisely those external edges that are incident to a vertex in $\mathcal{B}^N$ are occupied by $f$. 
 \item The $2N$ vertices in $\mathcal{L}^N\cup\mathcal{R}^N$ have degree 0 or 1.
 \item All other vertices of $G^N$ have degree 2.
 \item A path in $f$ neither connects two vertices of $\mathcal{L}^N$ nor two vertices of $\mathcal{R}^N$.
\end{enumerate}
\end{Def}

\begin{figure}[tbh]
\centering
\includegraphics[width=.9\textwidth]{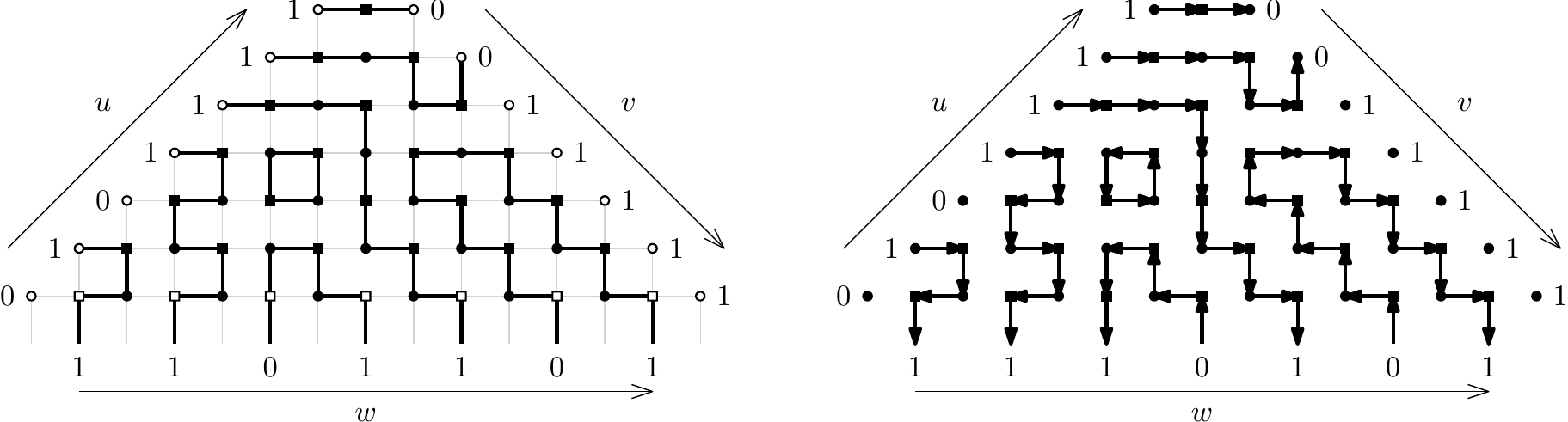}
\caption{A TFPL of size $7$ and an oriented TFPL of size $7$.}
\label{Fig:TFPLexample}
\end{figure}

An example of a TFPL is given in Figure~\ref{Fig:TFPLexample}. A \emph{cell of $f$ }is a cell of $G^N$ together with those of its surrounding edges that are occupied by $f$. 
To each TFPL of size $N$ is assigned a triple of words of length $N$.

\begin{Def}\label{Def:TFPLboundary}
Let $f$ be a TFPL of size $N$. The triple $(u,v;w)$ of words of length $N$ is assigned to $f$ as follows: 
\begin{enumerate}
 \item For $i=1,\dots,N$ set $u_i=1$ if the vertex $L_i\in\mathcal{L}^N$ has degree $1$ and $u_i=0$ otherwise. 
 \item For $i=1,\dots,N$ set $v_i=0$ if the vertex $R_i\in\mathcal{R}^N$ has degree $1$ and $v_i=1$ otherwise.
 \item For $i=1,\dots,N$ set $w_i=1$ if in $f$ the vertex $B_i\in\mathcal{B}^N$ is connected with a vertex in $\mathcal{L}^N$ or with a vertex $B_h$ for an $h<i$ and $w_i=0$ otherwise.  
\end{enumerate}
The triple $(u,v;w)$ is said to be the \emph{boundary} of $f$. Furthermore, the set of TFPLs with boundary $(u,v;w)$ is denoted by $T_{u,v}^w$ and its cardinality by $t_{u,v}^w$.
\end{Def}

For example, the triple \hbox{$(0101111,0011111;1101101)$} is the boundary of the TFPL depicted in Figure~\ref{Fig:TFPLexample}.
The definitions of both a TFPL and its boundary contain global conditions. Those can be omitted when adding an orientation to each edge of a TFPL.

\begin{Def}[Oriented triangular fully packed loop configuration]\label{Def:Oriented_TFPL}
An \textnormal{oriented TFPL} of size $N$ is a TFPL of size $N$ together with an orientation of its edges such that
the edges attached to $\mathcal{L}^N$ are outgoing, the edges attached to
$\mathcal{R}^N$ are incoming and all other vertices of $G^N$ are incident to an incoming and an outgoing edge.
\end{Def}

In Figure~\ref{Fig:TFPLexample}, an example of an oriented TFPL of size $7$ is given. In the underlying TFPL of an oriented TFPL condition~(4) can be omitted because the required orientations 
of the edges attached to a vertex of the left or right boundary prevent paths from returning to the respective boundary. 

\begin{Def}\label{Def:OrientedTFPLboundary} An oriented TFPL $f$ has boundary $(u,v;w)$ if the following hold:
\begin{enumerate}
 \item If the vertex $L_i\in\mathcal{L}^N$ has out-degree $1$ then $u_i=1$. Otherwise, $u_i=0$.
 \item If the vertex $R_i\in\mathcal{R}^N$ has in-degree $1$ then $v_i=0$. Otherwise, $v_i=1$. 
 \item If the external edge attached to the vertex $B_i\in\mathcal{B}^N$ is outgoing then $w_i=1$. Otherwise, $w_i=0$.
\end{enumerate} 
\end{Def}

While $u$ and $v$ coincide with the respective boundary word in the underlying ordinary TFPL this is not the case for $w$. 
Instead of the connectivity of the paths $w$ encodes the local orientation of the edges. Only in the case when in an oriented TFPL all paths between two vertices $B_i$ and $B_j$ of $\mathcal{B}^N$ 
are oriented from $B_i$ to $B_j$ if $i<j$ the boundary word $w$ coincides with the respective boundary word of the underlying TFPL. 
Hence, the \textit{canonical orientation} of a TFPL is defined as the orientation of the edges of the TFPL that satisfies the conditions in Definition~\ref{Def:Oriented_TFPL} 
and in addition that each path between two vertices $B_i,B_j\in\mathcal{B}^N$ is oriented from $B_i$ to $B_j$ if $i<j$ and that all closed paths are oriented clockwise.\\  

A triple $(u,v;w)$ that is the boundary of an ordinary or an oriented TFPL has to fulfill the following conditions: $\vert u\vert_0=\vert v\vert_0=\vert w\vert_0$, $u\leq w$, $v\leq w$ and $d(w)-d(u)-d(v)\geq 0$. 
These conditions were proved in \cite{CKLN,Thapper,TFPL}. The last condition gives rise to the following definition:

\begin{Def}[\cite{TFPL}]
Let $u,v,w$ be words of length $N$. Then the excess of $u,v,w$ is defined as
\begin{equation}
 exc(u,v;w)=d(w)-d(u)-d(v).
\notag\end{equation}
If $exc(u,v;w)=k$ then both an ordinary and an oriented TFPL with boundary $(u,v;w)$ are said to be of excess $k$.
\end{Def}

In \cite{TFPL}, the following interpretation of the excess of $u,v,w$ in terms of numbers of occurrences of certain local configurations in an oriented TFPL with boundary $(u,v;w)$ is proved: 

\begin{Prop}[{\cite[Theorem 4.3]{TFPL}}]\label{Prop:Comb_Int_Exc}
Let $f$ be an oriented TFPL with boundary $(u,v;w)$. Then
\begin{equation}\label{Excess}
\textnormal{exc}(u,v;w)= \vcenter{\hbox{\includegraphics[width=.011\textwidth]{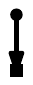}}}+ 
\vcenter{\hbox{\includegraphics[width=.011\textwidth]{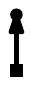}}}+\vcenter{\hbox{\includegraphics[width=.0825\textwidth]{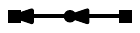}}}+
\vcenter{\hbox{\includegraphics[width=.0825\textwidth]{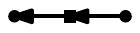}}}+\vcenter{\hbox{\includegraphics[width=.04675\textwidth]{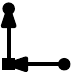}}}+
\vcenter{\hbox{\includegraphics[width=.04675\textwidth]{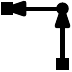}}}+\vcenter{\hbox{\includegraphics[width=.04675\textwidth]{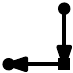}}}+
\vcenter{\hbox{\includegraphics[width=.04675\textwidth]{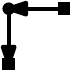}}}
\end{equation}
where by $\vcenter{\hbox{\includegraphics[width=.011\textwidth]{Expose5}}}$, $\vcenter{\hbox{\includegraphics[width=.011\textwidth]{Expose6}}}$, etc. the numbers of occurrences of the local configurations
$\vcenter{\hbox{\includegraphics[width=.011\textwidth]{Expose5}}}$, $\vcenter{\hbox{\includegraphics[width=.011\textwidth]{Expose6}}}$, etc. in $f$ are denoted.
\end{Prop}

\subsection{Blue-red path tangles}

In this subsection an alternative representation of oriented TFPLs is introduced, namely \textit{blue-red path tangles}. They came up in \cite{TFPL} and are crucial for the proofs given in this article.
Throughout this subsection, let $u,v,w$ be words of length $N$ such that $\vert u\vert_0=\vert v\vert_0=\vert w\vert_0=N_0$, $\vert u\vert_1=\vert v\vert_1=\vert w\vert_1=N_1$, 
$u\leq w$, $v\leq w$ and $ d(u)+d(v)\leq d(w)$.\\

A blue-red path tangle consists of an $N_0$-tuple of non-intersecting \textit{blue} lattice paths and an $N_1$-tuple of non-intersecting \textit{red} lattice paths. 
The blue lattice paths use steps $(-1,1)$, $(-1,-1)$ and $(-2,0)$, whereas the red lattice paths use steps $(1,1)$, $(1,-1)$ and $(2,0)$.
Furthermore, neither a blue nor a red lattice path goes below the $x$-axis. 
The $k$-th blue lattice path of an $N_0$-tuple of non-intersecting blue lattice paths starts in a certain fixed vertex $D_k$ and ends in a certain fixed vertex $E_k$. The definitions of the vertices $D_k$ and $E_k$ 
solely depend on the positions of the $k$-th zeroes in $w$ and $u$ and are omitted here. Instead, 
the vertices $D_1,\dots,D_{N_0}$ and $E_1,\dots,E_{N_0}$ are indicated with an example in Figure~\ref{Fig:Pathtangle_exc_2}.  
In the following, the set of $N_0$-tuples of non-intersecting blue lattice paths $(P_1,P_2,\dots,P_{N_0})$ where $P_k$ is a path from $D_k$ to $E_k$ 
is denoted by $\mathcal{P}(u,w)$. 
On the other hand, the $\ell$-th red path of an $N_1$-tuple of non-intersecting red lattice paths starts in a certain fixed vertex $D'_\ell$ and ends in a certain fixed vertex $E'_\ell$. The definitions of the vertices 
$D'_\ell$ and $E'_\ell$ solely depend on the positions of the $\ell$-th ones in $w$ and $v$ and are omitted here. Instead, $D'_\ell$ and $E'_\ell$ are indicated with and example in Figure~\ref{Fig:Pathtangle_exc_2}. 
In the following, the set of $N_1$-tuples of non-intersecting red paths $(P'_1,P_2',\dots,P_{N_1}')$
where $P'_\ell$ is a path from $D_\ell'$ to $E_\ell'$ is denoted by $\mathcal{P}'(v,w)$. 
 
\begin{figure}[tbh]
\includegraphics[width=.75\textwidth]{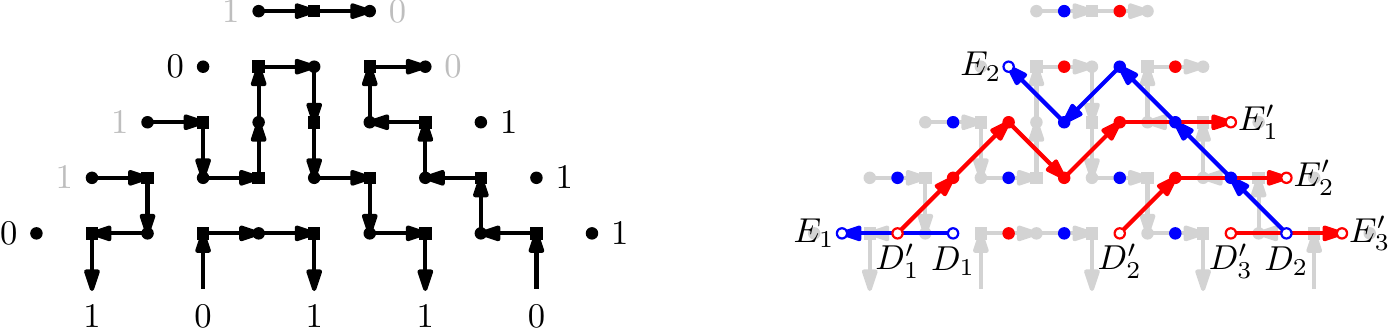}
\caption{An oriented TFPL with boundary $(01101,00111;10110)$ and its corresponding blue-red path tangle with boundary $(01101,00111;10110)$.}
\label{Fig:Pathtangle_exc_2}
\end{figure} 
 
\begin{Prop}[{\cite[Theorem 4.1]{TFPL}}]
The set of oriented TFPLs with boundary $(u,v;w)$ is in bijection with the set of pairs $(B,R)\in\mathcal{P}(u,w)\times\mathcal{P}'(v,w)$ that satisfy the two following conditions:
\begin{enumerate}
 \item No diagonal step of $R$ can cross a diagonal step of $B$.
 \item Each middle point of a horizontal step in $B$ (\textit{resp.} $R$) is used by a step in $R$ (\textit{resp.} $B$).
\end{enumerate}
The set of such configurations is denoted by \textnormal{BlueRed}$(u,v;w)$ and a configuration in \textnormal{BlueRed}$(u,v;w)$ is said to be a \textnormal{blue-red path tangle with boundary $(u,v;w)$}.
\end{Prop}

An example of an oriented TFPL and its corresponding blue-red path tangle is given in Figure~\ref{Fig:Pathtangle_exc_2}.  

\begin{proof} Here the bijection in \cite{TFPL} is repeated: let $f$ be an oriented TFPL of size $N$ and with boundary $(u,v;w)$. 
As a start \textit{blue vertices} are inserted in the middle of each horizontal edge of $G^N$ which has an odd vertex to its left and \textit{red vertices} are inserted in the middle of
each horizontal edge of $G^N$ which has an even vertex to its left. Next, \textit{blue edges} are inserted as indicated in the left part of Figure~\ref{Fig:Blue_red_edges} and \textit{red edges} are inserted as 
indicated in the right part of Figure~\ref{Fig:Blue_red_edges}. 
\begin{figure}[tbh]
\centering
\includegraphics[width=.75\textwidth]{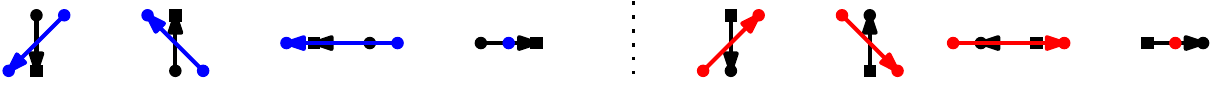}
\caption{From oriented TFPLs to blue-red path tangles.}
\label{Fig:Blue_red_edges}
\end{figure}
Then the blue vertices together with the blue edges give rise to an $N_0$-tuple of non-intersecting paths $B=(P_1,P_2,\dots,P_{N_0})$ in $\mathcal{P}(u,w)$ and the red vertices together with the red edges
give rise to an $N_1$-tuple of non-intersecting paths $R=(P'_1,P_2',\dots,P_{N_1}')$ in $\mathcal{P}'(v,w)$. The fact that 
no diagonal step of $R$ crosses a diagonal step of $B$ is equivalent to that there is a unique orientation of each vertical edge in 
$f$. On the other hand, the fact that each middle point of a horizontal step in $B$ (\textit{resp.} $R$) is used by a step in $R$ (\textit{resp.} $B$) is equivalent to that there is a unique orientation of each horizontal edge in $f$. 
Thus, $(B,R)\in\textnormal{BlueRed}(u,v;w)$.
\end{proof}

\section{Wieland drift}\label{Sec:WielandDrift}
The starting point of this section is the definition of Wieland gyration for 
fully packed loop configurations (FPLs) as introduced in \cite{Wieland1}. 
Wieland gyration is composed of local operations on all \emph{active} cells of an FPL: the active cells of an FPL can be chosen to be either all its odd cells or all its even cells.
Given an active cell $c$ of an FPL two cases have to be distinguished, namely whether $c$ contains precisely two edges of the FPL on opposite sides or not. If this is the case, Wieland gyration $\Wie$ 
leaves $c$ invariant. Otherwise, the effect of $\Wie$ on $c$ is that edges and non-edges of the FPL are exchanged. In Figure~\ref{Wieland}, the action of $\Wie$ on an active cell is illustrated. 

\begin{figure}[tbh]
\includegraphics[width=.325\textwidth]{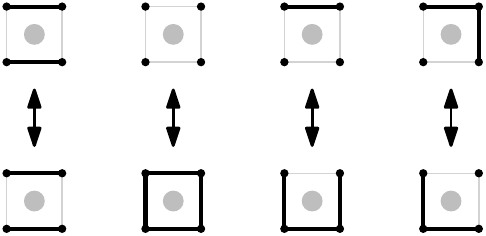}
\caption{Up to rotation, the action of $\Wie$ on an active internal cell of an FPL.}
\label{Wieland}
\end{figure} 

Also left- and right-Wieland drift will be composed of local operations on all \textit{active} cells of a TFPL. Similar to FPLs active cells of a TFPL are either chosen to be all its odd or all its even cells. 
Choosing all odd cells as active cells will lead to what will be defined as left-Wieland drift, whereas choosing all even cells as 
active cells will lead to what will be defined as right-Wieland drift. In the figures, the active cells of a TFPL will be indicated by gray circles.

\begin{Def}[Left-Wieland drift]\label{Def:LeftWieland}
Let $f$ be a TFPL with left boundary word $u$ and let $u^-$ be a word satisfying $u^-\stackrel{\mathrm{h}}{\longrightarrow}u$. 
The \emph{image of $f$ under left-Wieland drift with respect to $u^-$} is determined as follows:
\begin{enumerate}
 \item Insert a vertex $L'_i$ to the left of $L_i$ for $1\leq i\leq N$. Then run through the occurrences of ones in $u^-$: Let 
$\{i_1 < i_2 < \ldots < i_{N_1}\} = \{ i | u^-_i=1\}$.  
\begin{enumerate}
\item If $u_{i_j}$ is the $j$-th one in $u$, add a horizontal edge between $L'_{i_j}$ and $L_{i_j}$.
 \item If $u_{i_{j}-1}$ is the $j$-th one in $u$, add a vertical edge between $L'_{i_j}$ and $L_{{i_j}-1}$.
 \end{enumerate}
 \item Apply Wieland gyration to each odd cell of $f$.
 \item Delete all vertices in $\mathcal{R}^N$ and their incident edges.
 \end{enumerate}
After shifting the whole construction one unit to the right, one obtains the desired image $\WL_{u^-}(f)$. 
In the case $u^-=u$, simply write $\WL(f)$ and say the \emph{image of $f$ under left-Wieland drift}.
\end{Def}

\begin{figure}[tbh]
\begin{center}
\includegraphics[width=1\textwidth]{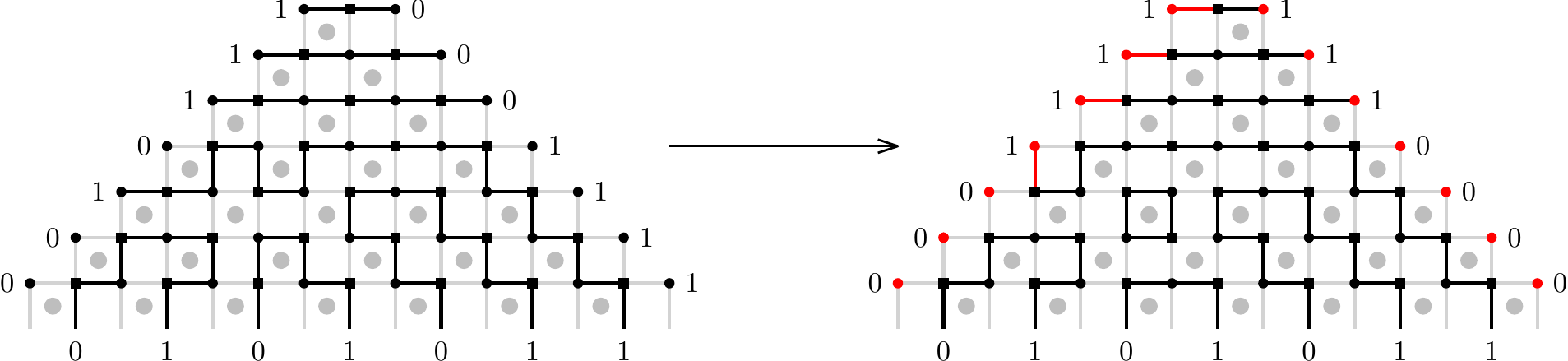}
\caption{A TFPL and its image under left-Wieland drift with respect to \hbox{$0001111$}.}
\label{Fig:WielandexampleTFPL}
\end{center}
\end{figure}

In Figure~\ref{Fig:WielandexampleTFPL}, an example for left-Wieland drift is given. The image of a TFPL with boundary $(u,v;w)$ under left-Wieland drift with respect to $u^-$ is again a TFPL and has boundary
$(u^-,v^+;w)$, where $v^+$ is a word satisfying $v\stackrel{\mathrm{v}}{\longrightarrow}v^+$, see \cite[Proposition 2.2]{WielandDrift}.\\

\emph{Right-Wieland drift} depends on a word $v^-$ satisfying $v^-\stackrel{\mathrm{v}}{\longrightarrow}v$ that encodes what happens along the right boundary of a TFPL with right boundary $v$ and is denoted by 
$\WR_{v^-}$ respectively $\WR$ if $v^-=v$. It is defined in an obvious way as the symmetric version of left-Wieland drift and it shall simply be illustrated with an example in Figure~\ref{InverseWielandExampleTFPL}. 

\begin{figure}[tbh]
\includegraphics[width=1\textwidth]{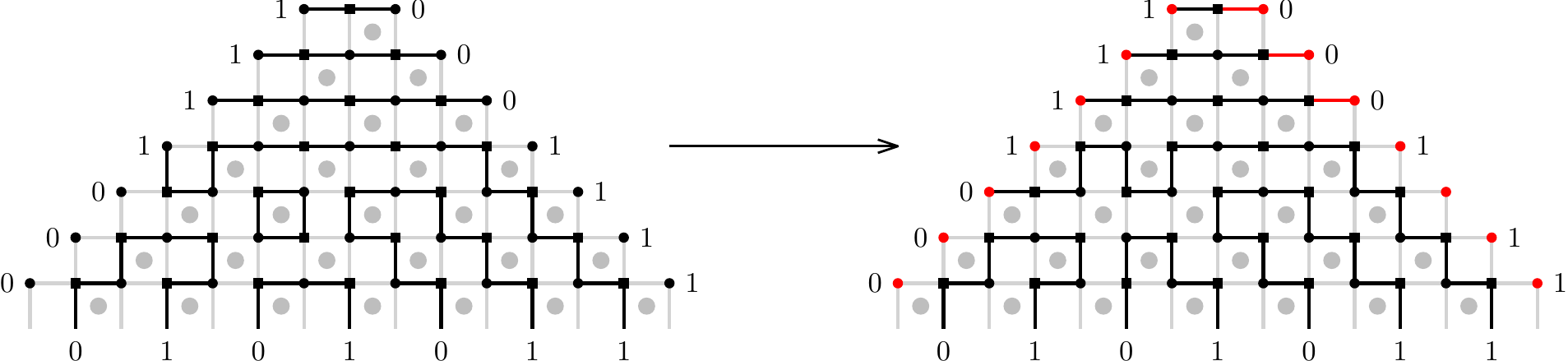}
\caption{A TFPL and its image under right-Wieland drift with respect to \hbox{$0001111$}.}
\label{InverseWielandExampleTFPL}
\end{figure}

The image of a TFPL with boundary $(u,v;w)$ under right-Wieland drift with respect to $v^-$ is a TFPL with boundary $(u^+,v^-;w)$ where $u^+$ is a word satisfying $u\stackrel{\mathrm{h}}{\longrightarrow}u^+$.\\

Given a TFPL with right boundary $v$ the effect of left-Wieland drift along the right boundary of the TFPL is inverted by right-Wieland drift with respect to $v$. On the other hand, given a TFPL with left 
boundary $u$ the effect of right-Wieland drift along the left boundary is inverted by left-Wieland drift with respect to $u$. Since Wieland gyration is an involution on each cell it follows:

\begin{Prop}[{\cite[Theorem 2]{WielandDrift}}]\label{Thm:WielandBijectiveTFPL} \begin{enumerate}
             \item Let $f$ be a TFPL with boundary $(u^+,v;w)$ and $u$ be a word such that $u\stackrel{\mathrm{h}}{\longrightarrow}u^+$. 
             Then 
             \[
             \WR_v(\WL_{u}(f))=f.
             \]
             \item Let $f$ be a TFPL with boundary $(u,v^+;w)$ and $v$ be a word such that 
             $v\stackrel{\mathrm{v}}{\longrightarrow}v^+$. Then 
             \begin{equation}
             \notag \WL_u(\WR_{v}(f))=f.
             \end{equation}
             
            \end{enumerate}

\end{Prop}

By Proposition~\ref{Thm:WielandBijectiveTFPL} a TFPL is invariant under left-Wieland drift if and only if it is invariant under right-Wieland drift.
Hence, a TFPL is said to be \textit{stable} if it is invariant under left-Wieland drift, whereas otherwise it is said to be \textit{instable}.
The set of stable TFPLs with boundary $(u,v;w)$ is denoted by $S_{u,v}^w$ and its cardinality by $s_{u,v}^w$.
In \cite{WielandDrift} it is shown that stable TFPLs can be characterized as follows:

\begin{Prop}[{\cite[Theorem 4]{WielandDrift}}]\label{Thm:StableTFPL}
A TFPL is stable if and only if it contains no edge of the form $\vcenter{\hbox{\includegraphics[width=.011\textwidth]{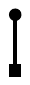}}}$. Such an edge is said to be a \textnormal{drifter}.
\end{Prop}

Note that by Proposition~\ref{Prop:Comb_Int_Exc} a TFPL of excess $k$ exhibits at most $k$ drifters.\\

Given a TFPL $f$ the sequence $(\WL^m(f))_{m\geq 0}$ is eventually periodic since there are only finitely many TFPLs of a fixed size. The length of its period is in fact always $1$.

\begin{Prop}[{\cite[Theorem 3]{WielandDrift}}]\label{Thm:EventuallyStable}
Let $f$ be a TFPL of size $N$. Then $\WL^{2N-1}(f)$ is stable, so that the following holds for all $m\geq 2N-1$:
\[
\WL^m(f)=\WL^{2N-1}(f).
\]
The same holds for right-Wieland drift.
\end{Prop}

\begin{figure}[tbh]
\begin{center}
\includegraphics[width=.75\textwidth]{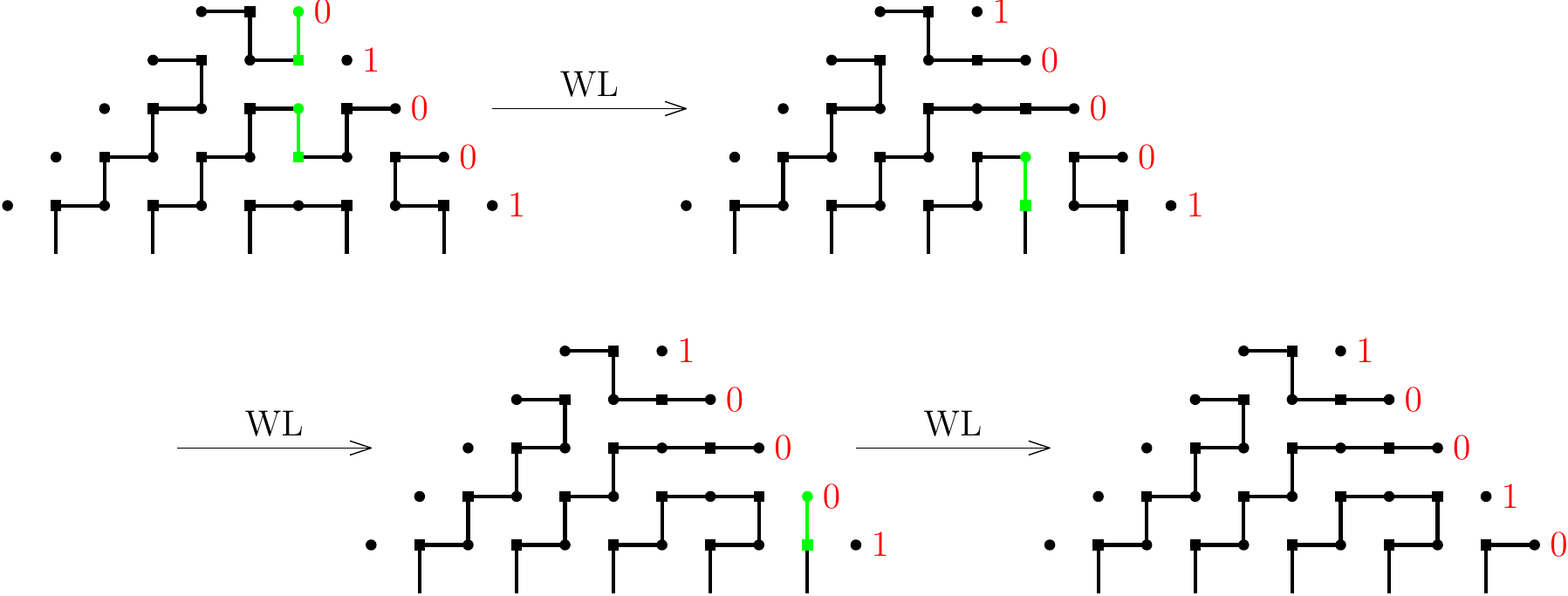}
\caption{A TFPL and its images under left-Wieland drift.}
\label{Fig:Eventually_Stable_Example_exc_2}
\end{center}
\end{figure}

In Figure~\ref{Fig:Eventually_Stable_Example_exc_2} an example of a TFPL and its images under left-Wieland drift is given. There a stable TFPL is obtained after the third iteration of 
left-Wieland drift. From now on, for an instable TFPL $f$ denote by $L=L(f)$ the positive integer $L$ such that 
$\operatorname{WL}^{\ell}(f)$ is instable for each $0\leq \ell\leq L$ and $\operatorname{WL}^{L+1}(f)$ is stable and by $R=R(f)$ the positive integer such that 
$\operatorname{WR}^{r}(f)$ is instable for each $0\leq r \leq R$ and $\operatorname{WR}^{R+1}(f)$ is stable.

\begin{Def}[$\operatorname{Path}(f)$, $\operatorname{Left}(f)$, $\operatorname{Right}(f)$]
Let $f$ be a TFPL.
The \textnormal{path} of $f$ -- denoted by Path($f$) -- is the sequence of all TFPLs that can be reached by an iterated application of left- respectively right-Wieland drift to $f$ that is 
\begin{equation}
\operatorname{Path}(f)=\left(\operatorname{WR}^{R+1}(f),\dots,\operatorname{WR}(f),f,\operatorname{WL}(f),\dots,\operatorname{WL}^{L+1}(f)\right). 
\notag\end{equation}
Furthermore, the stable TFPL $\operatorname{WR}^{R+1}(f)$ is denoted by $\operatorname{Right}(f)$ and the stable TFPL $\operatorname{WL}^{L+1}(f)$ by $\operatorname{Left}(f)$.
\end{Def}

When $v^{\ell}$ denotes the right boundary of $\operatorname{WL}^{\ell}(f)$ for each $0\leq \ell\leq L+1$ and $\lambda'$ denotes the conjugate of a Young diagram $\lambda$ then the sequence 
\begin{equation}
\lambda(v)'=\lambda(v^0)'\subseteq\lambda(v^1)'\subseteq\cdots\subseteq\lambda(v^L)'\subseteq\lambda(v^{L+1})' 
\label{Eq:Sequence_v}\end{equation}
gives rise to a semi-standard Young tableau of skew shape $\lambda(v^{L+1})/\lambda(v)'$ with entries $1,2,\dots,L+1$. On the other hand, when 
$u^{r}$ denotes the left boundary of $\operatorname{WR}^{r}(f)$ for each $0\leq r\leq R+1$ then the sequence 
\begin{equation}
\lambda(u)=\lambda(u^0)\subseteq\lambda(u^1)\subseteq\cdots\subseteq\lambda(u^R)\subseteq\lambda(u^{R+1}). 
\label{Eq:Sequence_u}\end{equation}
gives rise to a semi-standard Young tableau of skew shape $\lambda(u^{R+1})/\lambda(u)$. 

It will be shown that for an instable TFPL $f$ with boundary $(u,v;w)$ of excess at most $2$ \textbf{precisely one} of the following cases applies:
\begin{enumerate}
 \item the sequence in (\ref{Eq:Sequence_v}) corresponds to a semi-standard Young tableau in $G_{\lambda(v)',\lambda(v^+)'}$;
 \item the sequence in (\ref{Eq:Sequence_u}) corresponds to a semi-standard Young tableau in $G_{\lambda(u),\lambda(u^+)}$;
 \item neither the sequence in (\ref{Eq:Sequence_v}) corresponds to a semi-standard Young tableau in $G_{\lambda(v)',\lambda(v^+)'}$ 
 nor the sequence in (\ref{Eq:Sequence_u}) corresponds to a semi-standard Young tableau in $G_{\lambda(u),\lambda(u^+)}$.
\end{enumerate}

In the bijective proof of Theorem~\ref{Thm:Expressing_in_stable_TFPLs_exc_2} an instable TFPL $f$ with boundary $(u,v;w)$ of excess at most $2$ will be associated with the triple consisting of the empty semi-standard 
Young tableau of skew shape $\lambda(u)/\lambda(u)$, the stable TFPL $\operatorname{Left}(f)$ and the semi-standard Young tableau corresponding to the sequence 
in (\ref{Eq:Sequence_v}) if the latter is an element of $G_{\lambda(v)',\lambda(v^+)'}$.
If the semi-standard Young tableau corresponding to the sequence in (\ref{Eq:Sequence_u}) is an element of $G_{\lambda(u),\lambda(u^+)}$ 
then $f$ will be associated with the triple consisting of the semi-standard Young tableau in $G_{\lambda(u),\lambda(u^{R+1})}$ corresponding to the previous sequence, the stable TFPL $\operatorname{Right}(f)$ and the 
empty semi-standard Young tableau of skew shape $\lambda(v)'/\lambda(v)'$. Finally, if neither the sequence in (\ref{Eq:Sequence_v}) 
corresponds to a semi-standard Young tableau in $G_{\lambda(v)',\lambda(v^+)'}$ nor the sequence in (\ref{Eq:Sequence_u}) corresponds to a semi-standard Young tableau in $G_{\lambda(u),\lambda(u^+)}$ 
then to $f$ moves are applied which transform and ultimately turn it into a stable TFPL with boundary $(u^+,v^+;w)$ for a $u^+>u$ and a $v^+>v$. These moves will be extracted from the effect of Wieland drift on instable 
TFPLs of excess at most $2$. The triple which will be associated with $f$ then consists of this stable TFPL, a semi-standard Young tableau in $G_{\lambda(v)',\lambda(v^+)'}$ and one in $G_{\lambda(u),\lambda(u^+)}$.\\

In the next section, the effect of Wieland drift on instable TFPLs of excess at most $2$ is studied.

\section{An alternative description of Wieland drift for TFPLs of excess at most 2}\label{Sec:Wieland_drift_TFPL_Excess_2}

The main contribution of this section is a description of the effect of Wieland drift on TFPLs of excess at most $2$ as a composition of moves. 
In Figure~\ref{Fig:Moves_Exc2_unoriented}, the moves which form the basis for that description are depicted. Recall that a TFPL of excess $k$ contains at most $k$ drifters.

\begin{figure}[tbh]
\includegraphics[width=.7\textwidth]{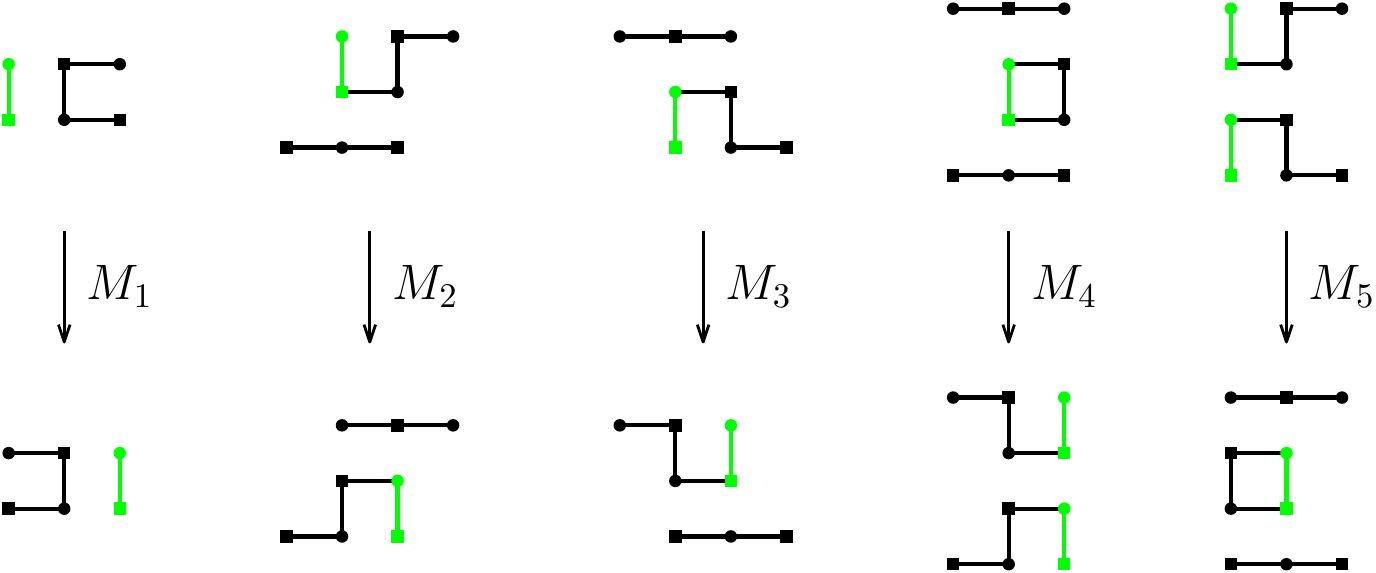}
\caption{The moves performed by left-Wieland drift in an instable TFPL of excess at most $2$.}
\label{Fig:Moves_Exc2_unoriented}
\end{figure}

\begin{Prop}\label{Prop:Moves_Exc2_unoriented} 
Let $f$ be an instable TFPL with boundary $(u,v;w)$ such that $\operatorname{exc}(u,v;w)\leq 2$. Furthermore, let $u^-$ be a word so that $u^-\stackrel{\mathrm{h}}{\longrightarrow}u$. Then the image of $f$ under left-Wieland drift with respect to $u^-$ is determined as follows:
\begin{enumerate}
\item if $R_i$ in $\mathcal{R}^N$ is incident to a drifter delete that drifter and add a horizontal edge incident to $R_{i+1}$ for $i=1,2,\dots,N-1$; denote the so-obtained TFPL by $f'$;
\item consider the columns of vertices of $G^N$ that contain a vertex, which is incident to a drifter in $f'$: let $\mathcal{I}=\{\ 2\leq i\leq 2N: \,\, a \,\, vertex\,\, of \,\, the \,\, i-th\,\, column\,\ is\,\, incident\,\, to \,\, a \,\, drifter\,\, in\,\, f'\}$, where the columns of $G^N$ 
are counted from left to right. 
\begin{enumerate}
\item If $\mathcal{I}=\{i<j\}$ apply a move in $\{M_1,M_2,M_3\}$ to the drifter incident to vertices of the $j$-th column and thereafter apply a move in $\{M_1,M_2,M_3\}$ to the drifter incident to vertices of the $i$-th column;
\item If $\mathcal{I}=\{i\}$ perform a move in $\{M_4,M_5\}$ or if this is not possible apply a move in $\{M_1,M_2,M_3\}$ to each of the drifters in $f'$ in the following order (if there are two drifters in $f'$): 
if the odd cell that contains the upper drifter is not of the form $\mathfrak{o}_9$ (see Figure~\ref{Fig:Cells_TFPL_Exc_2}) move the upper drifter first. Otherwise, move the lower drifter first.
\end{enumerate}
\item run through the occurrences of one in $u^-$: let $\{i_1<i_2<\cdots<i_{N_1}\}=\{i: u_i^-=1\}$. If $u_{i_j-1}$ is the $j$-th one in $u$ delete the horizontal edge incident to $L_{i_j-1}$ and add a vertical edge incident to $L_{i_j}$ for $j=1,2,\dots, N_1$.
\end{enumerate}
\end{Prop}

In Figure~\ref{Fig:Example_Wieland_Exc_2} a TFPL of excess $2$ with two drifters and its image under left-Wieland drift are depicted. The two drifters in the original TFPL have the same $x$-coordinate and the odd 
cell that contains the upper drifter is of the form $\mathfrak{o}_9$. 
Now, by left-Wieland drift the move $M_1$ is applied to the lower drifter before the move $M_2$ is applied to the other drifter. The rest of the TFPL is preserved.

\begin{figure}[tbh]
\centering
\includegraphics[width=.6\textwidth]{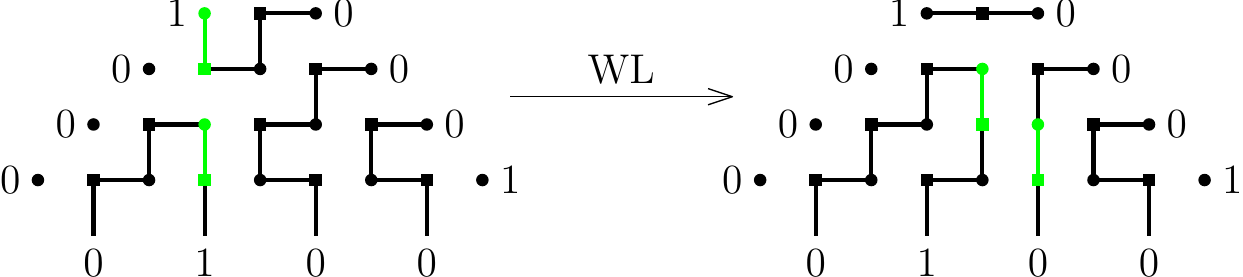}
\caption{A TFPL of excess $2$ with two drifters and its image under left-Wieland drift.}
\label{Fig:Example_Wieland_Exc_2}
\end{figure}

In the proof of Proposition~\ref{Prop:Moves_Exc2_unoriented} the effect of left-Wieland drift will be checked cell by cell. From the set of cells that can occur in a TFPL of excess at most $2$ the following cells can be 
excluded:

\begin{Lemma}\label{Lemma:Excluded_cells_exc_2}
In a TFPL of excess at most $2$, none of the following cells can occur:
\begin{center}
\includegraphics[width=.425\textwidth]{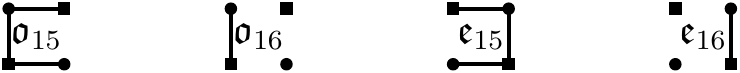}
\end{center}
\end{Lemma}
 
Since the proofs in this section work by studying the cells of a TFPL it is convenient to fix notations for all the odd and even cells that can occur in a 
TFPL. In total, there are 16 different odd and 16 different even \textit{internal} cells -- that are cells which are not external -- that can occur in a TFPL. 
By Lemma~\ref{Lemma:Excluded_cells_exc_2} fourteen of those odd and fourteen of those even internal cells that can occur in a TFPL of excess at most $2$. 
The odd respectively even cells that can occur in a TFPL of excess at most $2$ will be numbered by $1$ up to $14$ and are listed in Figure~\ref{Fig:Cells_TFPL_Exc_2}, whereas the two excluded odd respectively 
even internal cells will be numbered by $15$ and $16$ as indicated in Lemma~\ref{Lemma:Excluded_cells_exc_2}.

\begin{figure}[tbh]
\includegraphics[width=.75\textwidth]{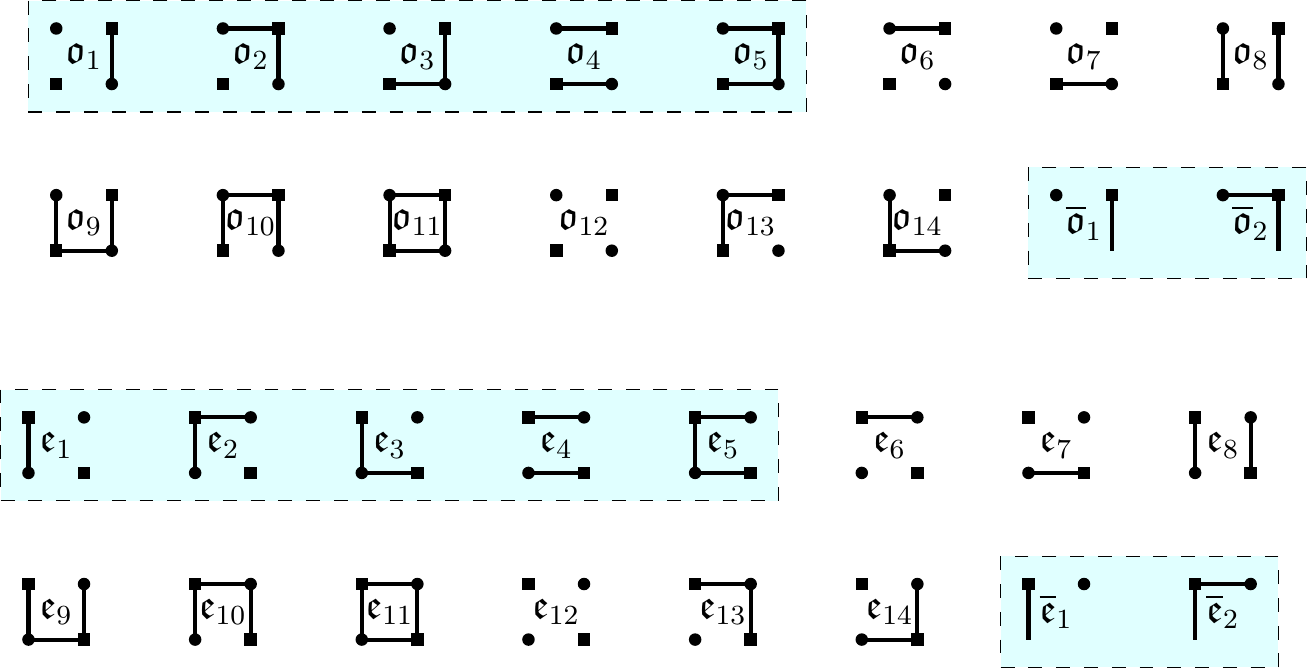}
\caption{The cells of a TFPL of excess at most $2$ where the sets \hbox{$\mathfrak{O}=\{\mathfrak{o}_1,\mathfrak{o}_2,\mathfrak{o}_3,\mathfrak{o}_4,\mathfrak{o}_5,\overline{\mathfrak{o}}_1,\overline{\mathfrak{o}}_2\}$} 
and $\mathfrak{E}=\{\mathfrak{e}_1,\mathfrak{e}_2,\mathfrak{e}_3,\mathfrak{e}_4,\mathfrak{e}_5,\overline{\mathfrak{e}}_1,\overline{\mathfrak{e}}_2\}$ are indicated.}
\label{Fig:Cells_TFPL_Exc_2}
\end{figure}

\begin{proof}
First, let $f$ be a TFPL that contains a cell $c$ that coincides with $\mathfrak{o}_{15}$. In $f$ together with its canonical orientation the oriented edges of $c$ then give rise to two configurations that 
are counted by the excess, see Proposition~\ref{Prop:Comb_Int_Exc}. Additionally, the right vertex of the horizontal edge of $c$ which is oriented from right to left either is adjacent to the vertex to its right or is 
incident to a drifter. Thus, the TFPL $f$ together with its canonical orientation contains at least three configurations that are counted by the excess. 
For the same reasons, a TFPL of excess at most $2$ cannot contain the third cell in the list. 

Now, let $f$ be a TFPL that contains a cell $c$ that coincides with $\mathfrak{o}_{16}$. Then both the top and the bottom rightmost vertex of $c$ have to be incident to a drifter. 
Therefore, $f$ contains at least three drifters and therefore has to be of excess at least $3$. 
By the same argument, the fourth cell in the list cannot occur in a TFPL of excess at most $2$.
\end{proof}

In the following, to distinguish between the cells of a TFPL and the cells of its image under left-Wieland drift given a cell $c$ of $G^N$ it is written $c$ when it is referred to the cell $c$ of the TFPL and 
$c'$ when it is referred to the cell $c$ of the image of the TFPL under left-Wieland drift. 
When the cells of a TFPL and of its image under left-Wieland drift are compared it has to be kept in mind that in the last step of left-Wieland drift the whole configuration is shifted one unit to the right. 
For that reason, for each odd cell $o$ of a TFPL and the even cell $e$ to the right of $o$ the following holds
when \textbf{disregarding} the distinction between odd and even vertices: 

\begin{center}
\fbox{$e'=\Wie(o)$}
\end{center}
\vspace{.3cm}

The odd cells $\mathfrak{O}=\{\mathfrak{o}_1,\mathfrak{o}_2,\mathfrak{o}_3,\mathfrak{o}_4,\mathfrak{o}_5,\overline{\mathfrak{o}}_1,\overline{\mathfrak{o}}_2\}$
and the even cells $\mathfrak{E}=\{\mathfrak{e}_1,\mathfrak{e}_2,\mathfrak{e}_3,\mathfrak{e}_4,\mathfrak{e}_5,\overline{\mathfrak{e}}_1,\overline{\mathfrak{e}}_2\}$ play a special role in the context of Wieland drift.

\begin{Lemma}[\cite{WielandDrift}]\label{Lemma:Effected_cells_exc_2}
Let $f$ be a TFPL, $o$ an odd cell of $f$ and $e$ the even cell to the right of $o$. If no vertex of $o$ and $e$ is incident to a drifter, 
then 
\begin{equation}
(o,e)\in\{(\mathfrak{o}_1,\mathfrak{e}_5), (\mathfrak{o}_2, \mathfrak{e}_3), (\mathfrak{o}_3,\mathfrak{e}_2), (\mathfrak{o}_4,\mathfrak{e}_4), (\mathfrak{o}_5,\mathfrak{e}_1), 
(\overline{\mathfrak{o}}_1,\overline{\mathfrak{e}}_2), (\overline{\mathfrak{o}}_2,\overline{\mathfrak{e}}_1)\}.
\notag\end{equation}
In particular, $e'=e$ in that case.
\end{Lemma}

To study the effect of left-Wieland drift on the whole TFPL it suffices to study its effect on the even cells of a TFPL. That is because edges 
of a TFPL that are not edges 
of an even cell have to be incident to a vertex in $\mathcal{L}^N$ and the effect of left-Wieland drift on these edges immediately follows from the definition of left-Wieland drift. 
To be more precise, in the image of a TFPL under left-Wieland drift all edges incident to a vertex in $\mathcal{L}^{N}$ have to be horizontal edges. 
By Lemma~\ref{Lemma:Effected_cells_exc_2}, to determine the effect of left-Wieland drift on a TFPL it suffices to determine its effect on the one hand on all even cells of the TFPL whereof a vertex is incident to a 
drifter and on the other hand on all even cells where the odd cells to their left contain a drifter.\\

Now, given a drifter $\mathfrak{d}$ in an instable TFPL there are at most three even cells whereof a vertex is incident to $\mathfrak{d}$ and there is at most one even 
cell such that the odd cell to its left contains $\mathfrak{d}$. In Figure~\ref{Fig:Surrounding_Cells_Exc_2}, these four even cells together with the odd cells to their left are depicted. 
Note that all four such even cells exist if and only if $\mathfrak{d}$ is not incident to a vertex in $\mathcal{L}^N\cup\mathcal{R}^N$.
From now on, these even cells and the odd cells to their left are denoted as indicated in Figure~\ref{Fig:Surrounding_Cells_Exc_2}. 

\begin{figure}[tbh]
\centering
\includegraphics[width=.14\textwidth]{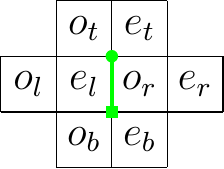} 
\caption{The even cells whereof a vertex is incident to a fixed drifter and the even cell such that the odd cell to its left contains the fixed drifter.}
\label{Fig:Surrounding_Cells_Exc_2}
\end{figure} 

Given a drifter in a TFPL of excess at most $2$ the effect of left-Wieland drift on the cell $e_t$, $e_l$ respectively $e_b$ can be uniformly described as long as $o_t$, $o_l$ respectively $o_b$ does not contain 
a drifter and is not in $\{\mathfrak{o}_6,\mathfrak{o}_7,\mathfrak{o}_{12}\}$:

\begin{Lemma}\label{Lemma:Wieland_boundary_cells_exc_2}
Let $f$ be an instable TFPL of excess at most $2$, $o$ an odd cell of $f$ and $e$ the even cell to the right of $o$. If $o$ does not contain a drifter, a vertex of $e$ is incident to a 
drifter and $o\notin \{\mathfrak{o}_6$, $\mathfrak{o}_7$, $\mathfrak{o}_{12}\}$, then $e'$ and $e$ coincide with the following sole exceptions:
\begin{enumerate}
 \item if in $e$ there is a drifter, then in $e'$ there is none,
 \item if the top left vertex of $e$ is incident to a drifter, then there is no horizontal edge between the two top vertices of $e$ but there is one between the two top vertices of $e'$,
 \item if the bottom left vertex of $e$ is incident to a drifter, then there is no horizontal edge between the two bottom vertices of $e$ but there is one between the two bottom vertices of $e'$.
\end{enumerate}
\end{Lemma}

\begin{proof} If $o$ and $e$ are external cells such that a vertex of $e$ is incident to a drifter, then $o=\overline{\mathfrak{o}}_1$, $e=\overline{\mathfrak{e}}_1$ and $e'=\overline{\mathfrak{e}}_2$. 
In that case, $e'$ and $e$ coincide with the sole exception that there is no horizontal edge between the two top vertices of $e$ whereas there is one between the two top vertices of $e'$. 
Suppose that $o$ and $e$ are internal cells such that $o$ does not contain a drifter, a vertex of $e$ is incident to a drifter and $o$ is not in $\{\mathfrak{o}_6$, $\mathfrak{o}_7$, $\mathfrak{o}_{12}\}$. 
Then $(o,e)$ can only occur as part of one of the following pairs:
\begin{center}
\includegraphics[width=1\textwidth]{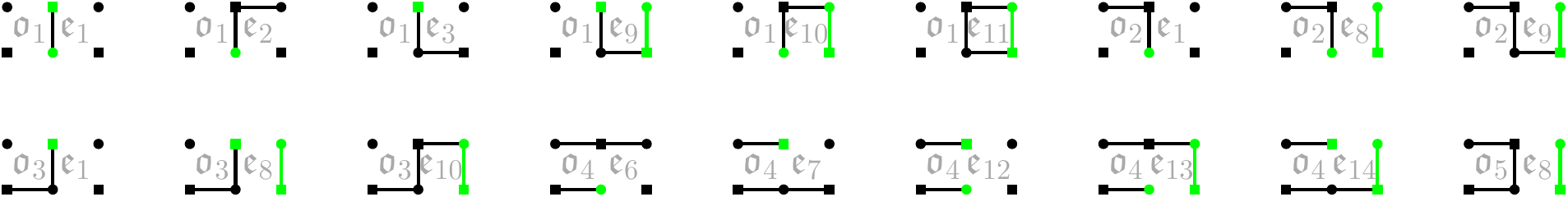} 
\end{center}
Now, $e'=\mathfrak{e}_5$ if $o=\mathfrak{o}_1$, $e'=\mathfrak{e}_3$ if $o=\mathfrak{o}_2$, $e'=\mathfrak{e}_2$ if $o=\mathfrak{o}_3$, 
$e'=\mathfrak{e}_4$ if $o=\mathfrak{o}_4$ and $e'=\mathfrak{e}_1$ if $o=\mathfrak{o}_5$. 
It can easily be checked that in any case $e'$ and $e$ satisfy the assertions.
\end{proof}

In the following, separate proofs for each case in Proposition~\ref{Prop:Moves_Exc2_unoriented} will be given. 

\begin{proof}[Proof of Proposition~\ref{Prop:Moves_Exc2_unoriented}(1)]
Let $f$ be an instable TFPL of excess at most $2$ that contains precisely one drifter $\mathfrak{d}$.
First, the case when $\mathfrak{d}$ is incident to a vertex $R_i$ in $\mathcal{R}^N$ is considered. In that case the cells $o_l$, $e_l$, $o_b$ and $e_b$ exist. Furthermore, both $o_l$ and $o_b$ do not contain a  
drifter and are not in $\{\mathfrak{o}_6,\mathfrak{o}_7,\mathfrak{o}_{12}\}$ because in $f$ there is only one drifter. Thus, by Lemma~\ref{Lemma:Wieland_boundary_cells_exc_2} on the one hand $e'_l$ 
and $e_l$ coincide with the sole exception that in $e'_l$ there is no drifter whereas in $e_l$ there is one and on the other hand 
$e'_b$ and $e_b$ coincide with the sole exception that in $e'_b$ the two top vertices are adjacent whereas in $e_b$ they are not. 
By Lemma~\ref{Lemma:Effected_cells_exc_2} the effect of left-Wieland drift on $f$ is that the drifter incident to $R_i$ is replaced by a horizontal edge incident to $R_{i+1}$ while the rest of $f$ is preserved.

It remains to consider the case when $\mathfrak{d}$ is not incident to a vertex in $\mathcal{R}^N$. In that case the cells $o_r$, $e_r$, $o_b$ and $e_b$ of $f$ have to exist. 
Since $f$ contains precisely one drifter $o_r\in\{\mathfrak{o}_{8},\mathfrak{o}_{9},\mathfrak{o}_{10},\mathfrak{o}_{11}\}$ by Lemma~\ref{Lemma:Excluded_cells_exc_2}.
It will be proceeded by treating each of the four possible cases for $o_r$ separately. 

First, the case when $o_r=\mathfrak{o}_8$ is regarded. In that case, $e_r=\mathfrak{e}_5$ because $f$ contains precisely one drifter. 
Furthermore, $\Wie(\mathfrak{o}_8)=\mathfrak{o}_8$ and therefore $e'_r=\mathfrak{e}_8$. On the other hand, $e'_b$ and $e_b$ coincide with the sole exception that 
the two top vertices in $e'_b$ are adjacent whereas in $e_b$ they are not by Lemma~\ref{Lemma:Wieland_boundary_cells_exc_2}. If the cells $o_t$, $e_t$, $o_l$ and $e_l$ exist, then both $o_t$ and $o_l$ cannot be in 
$\{\mathfrak{o}_6,\mathfrak{o}_7,\mathfrak{o}_{12}\}$ and for that reason $e'_t$ and $e_t$ coincide with the sole exception that the two bottom vertices in $e'_t$ are adjacent whereas 
in $e_t$ they are not and $e'_l$ and $e_l$ coincide with the sole exception that in $e'_l$ there is no drifter whereas in $e_l$ there is one by 
Lemma~\ref{Lemma:Wieland_boundary_cells_exc_2}. By Lemma~\ref{Lemma:Effected_cells_exc_2}, the effect of left-Wieland drift on $f$ is that the move $M_1$ is applied to $\mathfrak{d}$ while the rest of $f$ 
is preserved.

Next, the case when $o_r=\mathfrak{o}_9$ is considered. In that case, $e_r=\mathfrak{e}_2$, $o_b=\mathfrak{o}_7$ and $e_b=\mathfrak{e}_4$. 
Furthermore, $\Wie(\mathfrak{o}_9)=\mathfrak{o}_6$ that is $e'_r=\mathfrak{e}_6$ and $\Wie(\mathfrak{o}_7)=\mathfrak{o}_{10}$ that is $e'_b=\mathfrak{e}_{10}$. 
If the cells $o_l$, $e_l$, $o_t$ and $e_t$ exist, then neither $o_t$ nor $o_l$ is in $\{\mathfrak{o}_6,\mathfrak{o}_7,\mathfrak{o}_{12}\}$. By Lemma~\ref{Lemma:Wieland_boundary_cells_exc_2} and 
Lemma~\ref{Lemma:Effected_cells_exc_2} the effect of left-Wieland drift on $f$ is that the move $M_2$ is applied to $\mathfrak{d}$ while the rest of $f$ is preserved. 

Next, the case when $o_r=\mathfrak{o}_{10}$ is regarded. In that case, $o_t$, $e_t$, $o_l$ and $e_l$ exist. Furthermore, $e_r=\mathfrak{e}_3$, $o_t=\mathfrak{o}_6$ and $e_t=\mathfrak{e}_4$.
Therefore, $e'_r=\mathfrak{e}_7$ and $e'_t=\mathfrak{e}_9$. Since neither $o_b$ nor $o_l$ is in $\{\mathfrak{o}_6,\mathfrak{o}_7,\mathfrak{o}_{12}\}$ by Lemma~\ref{Lemma:Wieland_boundary_cells_exc_2} and 
Lemma~\ref{Lemma:Effected_cells_exc_2} the effect of left-Wieland drift on $f$ is that the move $M_3$ is applied to $\mathfrak{d}$ while the rest of $f$ is preserved.

Finally, the case when $o_r=\mathfrak{o}_{11}$ is checked. In that case $o_t$, $e_t$, $o_l$ and $e_l$ exist. Furthermore, $e_r=\mathfrak{e}_1$, $o_b=\mathfrak{o}_7$, $e_b=\mathfrak{e}_4$, 
$o_l=\mathfrak{o}_5$, $e_l=\mathfrak{e}_8$, $o_t=\mathfrak{o}_6$ and $e_t=\mathfrak{e}_{4}$. Therefore, $e'_r=\mathfrak{e}_{12}$, $e'_b=\mathfrak{e}_{10}$, 
$e'_l=\mathfrak{e}_{1}$ and $e'_t=\mathfrak{e}_{9}$. By Lemma~\ref{Lemma:Effected_cells_exc_2} the effect of left-Wieland drift on $f$ is that the move $M_4$ is applied to 
$\mathfrak{d}$ while the rest of $f$ is preserved.
\end{proof}

\begin{proof}[Proof of Proposition~\ref{Prop:Moves_Exc2_unoriented}(2a)] 
Let $f$ be a TFPL of excess $2$ that contains two drifters which are both incident to a vertex in $\mathcal{R}^N$. In the following, denote the two drifter in $f$ by $\mathfrak{d}$ and $\mathfrak{d}^\ast$ 
and let $R_i$ and $R_{i^\ast}$ be the vertices in $\mathcal{R}^N$ to which $\mathfrak{d}$ and $\mathfrak{d}^\ast$ are incident.
The cells $o_l$, $e_l$, $o_b$ and $e_b$ and the cells $o_l^\ast$, $e_l^\ast$, $o_b^\ast$ and $e_b^\ast$ exist. Furthermore, none of the cells $o_l$, $o_b$, $o_l^\ast$ and $o_b^\ast$ contains a drifter 
or is of the form $\mathfrak{o}_6$, $\mathfrak{o}_7$ or $\mathfrak{o}_{12}$. Therefore, by Lemma~\ref{Lemma:Wieland_boundary_cells_exc_2} and Lemma~\ref{Lemma:Effected_cells_exc_2} the effect of left-Wieland drift on $f$ 
is that both drifters are replaced by horizontal edges incident to $R_{i+1}$ and $R_{i^\ast+1}$ while the rest of $f$ remains unchanged. 
\end{proof}

\begin{proof}[Proof of Proposition~\ref{Prop:Moves_Exc2_unoriented}(2b)] 
Let $f$ be a TFPL of excess $2$ that contains two drifters $\mathfrak{d}$ and $\mathfrak{d}^{\ast}$ whereof $\mathfrak{d}$ is incident to a vertex $R_i$ in $\mathcal{R}^N$ and $\mathfrak{d}^\ast$ is not 
incident to a vertex in $\mathcal{R}^N$. Note that $f$ contains neither a cell of type $\mathfrak{o}_{11}$ nor of type $\mathfrak{e}_{11}$. That is because when adding the canonical orientation 
to $f$ such a cell would give rise to two local configurations that are counted by the excess which would imply that $f$ is of excess greater than $2$.  
As a start, suppose that no vertex of $o_l$ and $o_b$ is incident to $\mathfrak{d}^\ast$. In that case $e'_l$ and $e_l$ coincide with the sole exception that in 
$e'_l$ there is no drifter and $e'_b$ and $e_b$ coincide with the sole exception that in $e'_b$ the two top vertices are adjacent whereas in $e_b$ 
they are not by Lemma~\ref{Lemma:Wieland_boundary_cells_exc_2}. On the other hand, since $\mathfrak{d}^\ast$ is not incident to a vertex in $\mathcal{R}^N$ the cells $o_r^\ast$, $e_r^\ast$, $o_b^\ast$ and $e_b^\ast$ 
exist. Futhermore, no vertex of $o_r^\ast$, $e_r^\ast$, $o_b^\ast$ and $e_b^\ast$ is incident to $\mathfrak{d}$. If the cells $o_l^\ast$, $e_l^\ast$, $o_t^\ast$ and $e_t^\ast$ exist then also no vertex of these cells 
is incident to $\mathfrak{d}$. For those reasons, by analogous arguments as in the proof of Proposition~\ref{Prop:Moves_Exc2_unoriented}(1) the effect of left-Wieland drift on $f$ is that $\mathfrak{d}$ is replaced by a 
horizontal edge incident to $R_{i+1}$ before a unique move of $\{M_1,M_2,M_3\}$ is applied to $\mathfrak{d}^\ast$. The rest of $f$ is preserved by left-Wieland drift.

Now, if the bottom right vertex of $o_b$ is incident to $\mathfrak{d}^\ast$, then $o_r^\ast\in\{\mathfrak{o}_8,\mathfrak{o}_9,\mathfrak{o}_{10}\}$. If $o_r^\ast=\mathfrak{o}_8$, then $e_r^\ast=\mathfrak{e}_5$. 
Furthermore, $o_l$, $o_b$, $o_l^\ast$ and $o_b^\ast$ 
do not contain a drifter and are not in $\{\mathfrak{o}_6,\mathfrak{o}_7,\mathfrak{o}_{12}\}$. Thus, the effect of left-Wieland drift is that $\mathfrak{d}$ is replaced by a horizontal edge incident to $R_{i+1}$ before 
the move $M_1$ is applied to $\mathfrak{d}^\ast$ while the rest of $f$ is preserved. If $o_r^\ast=\mathfrak{o}_9$, then $e_r^\ast=\mathfrak{e}_2$, $o_b^\ast=\mathfrak{o}_7$ and $e_b^\ast=\mathfrak{e}_4$. Additionally, 
$o_l$, $o_b$ and $o_l$ do not contain a drifter and are not in $\{\mathfrak{o}_6,\mathfrak{o}_7,\mathfrak{o}_{12}\}$. Therefore, the effect of left-Wieland drift is that $\mathfrak{d}$ is replaced by a horizontal edge 
incident to $R_{i+1}$ before the move $M_2$ is applied to $\mathfrak{d}^\ast$ while the rest of $f$ is preserved. Finally, if $o_r^\ast=\mathfrak{o}_{10}$, then $e_r^\ast=\mathfrak{e}_3$, $o_b=\mathfrak{o}_6$ and 
$e_b=\mathfrak{e}_7$. Furthermore, $o_l$, $o_l^\ast$ and $o_b^\ast$ do not contain a drifter and are not in $\{\mathfrak{o}_6,\mathfrak{o}_7,\mathfrak{o}_{12}\}$. For those reasons, 
the effect of left-Wieland drift is that $\mathfrak{d}$ is replaced by a horizontal edge 
incident to $R_{i+1}$ before the move $M_3$ is applied to $\mathfrak{d}^\ast$ while the rest of $f$ is preserved.

Next, if $o_l$ contains $\mathfrak{d}^\ast$, then $o_l=\mathfrak{o}_{10}$, $e_l=\mathfrak{e}_9$, $o_t^\ast=\mathfrak{o}_6$, $e_t^\ast=\mathfrak{e}_4$, $o_b=\mathfrak{o}_4$, $e_b=\mathfrak{e}_7$ and 
\hbox{$o_b^\ast\notin\{\mathfrak{o}_6, \mathfrak{o}_7, \mathfrak{o}_{12}\}$}. Therefore, $e'_l=\mathfrak{e}_7$, $e'_t=\mathfrak{e}_9$, $e'_b=\mathfrak{e}_4$ and 
by Lemma~\ref{Lemma:Wieland_boundary_cells_exc_2} the cells $e^{\ast\prime}_b$ and $e_b^\ast$ coincide with the sole exception that in $e^{\ast\prime}_b$ there is an edge between the two top 
vertices whereas in $e_b^\ast$ there is none. For those reasons, the effect of left-Wieland drift on $f$ is that $\mathfrak{d}$ is replaced by a horizontal edge incident to $R_{i+1}$ before the move $M_3$ is applied to 
$\mathfrak{d}^\ast$ while the rest of $f$ is preserved.

Finally, if $o_b$ contains $\mathfrak{d}^\ast$, then $o_b\in\{\mathfrak{o}_8,\mathfrak{o}_9\}$.
If $o_b=\mathfrak{o}_8$, then $e_b=\mathfrak{e}_3$. Furthermore, none of the cells $o_l$, $o_l^\ast$ and $o_b^\ast$ is in 
$\{\mathfrak{o}_6, \mathfrak{o}_7, \mathfrak{o}_{12}\}$. Thus, the effect of left-Wieland drift on $f$ is that 
$\mathfrak{d}$ is replaced by a horizontal edge incident to $R_{i+1}$ before the move $M_1$ is applied to 
$\mathfrak{d}^\ast$ while the rest of $f$ is preserved. On the other hand, if $o_b=\mathfrak{o}_9$, then $e_b=\mathfrak{e}_1$, $o_b^\ast=\mathfrak{o}_7$ and $e_b^\ast=\mathfrak{e}_4$. Additionally, $o_l$ and $o_l^\ast$ 
do not contain a drifter and are not in $\{\mathfrak{o}_6, \mathfrak{o}_7, \mathfrak{o}_{12}\}$. Therefore, the effect of left-Wieland drift on $f$ is that $\mathfrak{d}$ is replaced by a horizontal edge incident to 
$R_{i+1}$ before the move $M_2$ is applied to $\mathfrak{d}^\ast$ while the rest of $f$ is preserved.
\end{proof}

\begin{proof}[Proof of Proposition~\ref{Prop:Moves_Exc2_unoriented}(2c)] 
Let $f$ be a TFPL of excess $2$ that contains two drifters whereof none is incident to a vertex in $\mathcal{R}^N$. In that case the cells $o_r$, $e_r$, $o_b$, $e_b$ and the cells 
$o_r^{\ast}$, $e_r^{\ast}$, $o_b^{\ast}$, $e_b^{\ast}$ exist.
Furthermore, both $o_r$ and $o_r^{\ast}$ have to be in $\{\mathfrak{o}_{8},\mathfrak{o}_{9},\mathfrak{o}_{10},\mathfrak{o}_{13},\mathfrak{o}_{14}\}$.
It is started with the case when no vertex of the cells $o_r$, $e_r$, $o_b$, $e_b$ is incident to $\mathfrak{d}^{\ast}$ and no vertex of the cells 
$o_r^{\ast}$, $e_r^{\ast}$, $o_b^{\ast}$, $e_b^{\ast}$ is incident to $\mathfrak{d}$. 
This implies that if the cells $o_l$, $e_l$, $o_t$ and $e_t$ exist then none of their vertices is incident to $\mathfrak{d}^\ast$ and if the cells $o_l^{\ast}$, $e_l^{\ast}$, $o_t^{\ast}$ and $e_t^{\ast}$ exist then 
none of their vertices is incident to $\mathfrak{d}$. Therefore, by the same arguments as in the proof of Proposition~\ref{Prop:Moves_Exc2_unoriented}(1) the effect of left-Wieland drift on $f$ is that simultaneously 
to each of the two drifters $\mathfrak{d}$ and $\mathfrak{d}^\ast$ a unique move in $\{M_1,M_2,M_3\}$ is applied while the rest of $f$ is conserved. Since the moves can be 
performed simultaneously they can be performed in the order stated in Proposition~\ref{Prop:Moves_Exc2_unoriented}(2).

It remains to study the case when a vertex of $o_r$, $e_r$, $o_b$ or $e_b$ is incident to $\mathfrak{d}^{\ast}$ or a vertex of $o_r^{\ast}$, $e_r^{\ast}$, $o_b^{\ast}$ or $e_b^{\ast}$ is incident to $\mathfrak{d}$. 
Hence, without loss of generality assume that a vertex of the cells $o_r$, $e_r$, $o_b$, $e_b$ that is not the top right vertex of $o_r$ 
is incident to the drifter $\mathfrak{d}^{\ast}$. Then $o_r$ does not equal $\mathfrak{o}_{14}$ and $o_r^{\ast}$ does not equal $\mathfrak{o}_{13}$. 

As a start, the case when the bottom right vertex of $o_b$ is incident to $\mathfrak{d}^{\ast}$ is considered. In that case $\mathfrak{d}$ and $\mathfrak{d}^\ast$ have the same $x$-coordinate 
and $\mathfrak{d}$ has the larger $y$-coordinates than $\mathfrak{d}^\ast$.
If the cells $o_t$, $e_t$, $o_l$ and $e_l$ exist then $o_t$ neither equals $\mathfrak{o}_7$ nor $\mathfrak{o}_{12}$. Furthermore, if $o_t=\mathfrak{o}_6$ then $e_t=\mathfrak{e}_4$, $o_r=\mathfrak{o}_{10}$ and 
$e_r=\mathfrak{e}_3$. Thus, $e'_t=\mathfrak{e}_9$ and $e'_r=\mathfrak{e}_7$. On the other hand, if $o_t$ does not equal $\mathfrak{o}_6$ then  
$e'_t$ and $e_t$ coincide with the sole exception that in $e'_t$ there is a horizontal edge between its two bottom vertices whereas in $e_t$ there is none 
by Lemma~\ref{Lemma:Wieland_boundary_cells_exc_2}. Since neither $o_l$ nor $o_l^\ast$ equals $\mathfrak{o}_6$, $\mathfrak{o}_7$ or $\mathfrak{o}_{12}$ the cells $e'_l$ and $e_l$ (\textit{resp.} 
$e^{\ast\prime}_l$ and $e_l^\ast$) coincide with the sole exception that in $e'_l$ (\textit{resp.} $e^{\ast\prime}_l$) there is no drifter
by Lemma~\ref{Lemma:Wieland_boundary_cells_exc_2}. 
Finally, $o_b^\ast$ does neither equal $\mathfrak{o}_6$ nor $\mathfrak{o}_{12}$ and if it equals $\mathfrak{o}_7$ then $e_b^\ast=\mathfrak{e}_4$, $o_r^\ast=\mathfrak{o}_9$, 
$e_r^\ast=\mathfrak{e}_2$, $e^{\ast\prime}_b=\mathfrak{e}_{10}$ and $e^{\ast\prime}_r=\mathfrak{e}_3$.

By Lemma~\ref{Lemma:Effected_cells_exc_2}, it remains to study the cells $o_r$, $e_r$, $o_b$, $e_b$,
$o_r^{\ast}$, $e_r^{\ast}$, $e_r'$, $e_b'$ and $e_r^{\ast\prime}$. A list of all possible configurations in the cells $o_r$, $e_r$, $o_b$, $e_b$,
$o_r^{\ast}$, $e_r^{\ast}$, $e_r'$, $e_b'$ and $e_r^{\ast\prime}$ is given in Table~\ref{Table:Drifter_bottom_right_o_b_exc_2}.

\begin{table}[tbh]
\begin{center}
\begin{tabular}{c!{\vrule width 1.35pt}c!{\vrule width 1pt}c!{\vrule width 1pt}c!{\vrule width 1pt}c!{\vrule width 1pt}c!{\vrule width 1pt}c!{\vrule width 1pt}c!{\vrule width 1pt}c!{\vrule width 1pt}c
!{\vrule width 1pt}c!{\vrule width 1pt}c!{\vrule width 1pt}c}
$o_r$ & $\mathfrak{o}_{8}$ & $\mathfrak{o}_{8}$ & $\mathfrak{o}_{8}$ & $\mathfrak{o}_{8}$ & $\mathfrak{o}_{8}$ & $\mathfrak{o}_{9}$ & $\mathfrak{o}_{9}$ & $\mathfrak{o}_{9}$ & 
$\mathfrak{o}_{10}$ & $\mathfrak{o}_{10}$ & $\mathfrak{o}_{10}$ & $\mathfrak{o}_{10}$ \notag\\
\hline
$e_r$ & $\mathfrak{e}_5$ & $\mathfrak{e}_5$ & $\mathfrak{e}_5$ & $\mathfrak{e}_5$ & $\mathfrak{e}_5$ & $\mathfrak{e}_2$ & $\mathfrak{e}_2$ & $\mathfrak{e}_2$ & $\mathfrak{e}_3$ & $\mathfrak{e}_3$ & $\mathfrak{e}_3$
& $\mathfrak{e}_3$ \notag\\
\hline
$o_r^\ast$ & $\mathfrak{o}_{8}$ & $\mathfrak{o}_{8}$ & $\mathfrak{o}_{9}$ & $\mathfrak{o}_{9}$ & $\mathfrak{o}_{10}$ & $\mathfrak{o}_{8}$ & $\mathfrak{o}_{9}$ & $\mathfrak{o}_{10}$ & $\mathfrak{o}_{8}$ & 
$\mathfrak{o}_{8}$ & $\mathfrak{o}_{9}$ & $\mathfrak{o}_{10}$ \notag\\
\hline
$e_r^\ast$ & $\mathfrak{e}_5$ & $\mathfrak{e}_5$ & $\mathfrak{e}_2$ & $\mathfrak{e}_2$ & $\mathfrak{e}_3$ & $\mathfrak{e}_5$ & $\mathfrak{e}_2$ & $\mathfrak{e}_3$ & $\mathfrak{e}_5$ & $\mathfrak{e}_5$ &
$\mathfrak{e}_2$ & $\mathfrak{e}_3$\notag\\
\hline
$o_b$ & $\mathfrak{o}_{1}$ & $\mathfrak{o}_{4}$ & $\mathfrak{o}_{1}$ & $\mathfrak{o}_{4}$ & $\mathfrak{o}_{6}$ & $\mathfrak{o}_{7}$ & $\mathfrak{o}_{7}$ & $\mathfrak{o}_{12}$ & 
$\mathfrak{o}_{1}$ & $\mathfrak{o}_{4}$ & $\mathfrak{o}_{4}$ & $\mathfrak{o}_{6}$ \notag\\
\hline
$e_b$ & $\mathfrak{e}_1$ & $\mathfrak{e}_{12}$ & $\mathfrak{e}_1$ & $\mathfrak{e}_{12}$ & $\mathfrak{e}_7$ & $\mathfrak{e}_6$ & $\mathfrak{e}_6$ & $\mathfrak{e}_4$ & $\mathfrak{e}_1$ & $\mathfrak{e}_{12}$ & 
$\mathfrak{e}_{12}$ & $\mathfrak{e}_7$\notag\\
\noalign{\hrule height 1.25pt}
$e'_r$ & $\mathfrak{e}_{8}$ & $\mathfrak{e}_{8}$ & $\mathfrak{e}_{8}$ & $\mathfrak{e}_{8}$ & $\mathfrak{e}_{8}$ & $\mathfrak{e}_{6}$ & $\mathfrak{e}_{6}$ & 
$\mathfrak{e}_{6}$ & $\mathfrak{e}_{7}$ & $\mathfrak{e}_{7}$ & $\mathfrak{e}_{7}$ & $\mathfrak{e}_{7}$ \notag\\
\hline
$e^{\ast\prime}_r$ & $\mathfrak{e}_{8}$ & $\mathfrak{e}_{8}$ & $\mathfrak{e}_{6}$ & $\mathfrak{e}_{6}$ & $\mathfrak{e}_{7}$ & $\mathfrak{e}_{8}$ & $\mathfrak{e}_{6}$ & 
$\mathfrak{e}_{7}$ & $\mathfrak{e}_{8}$ & $\mathfrak{e}_{8}$ & $\mathfrak{e}_{6}$ & $\mathfrak{e}_{7}$ \notag\\
\hline
$e'_b$ & $\mathfrak{e}_{15}$ & $\mathfrak{e}_{4}$ & $\mathfrak{e}_{15}$ & $\mathfrak{e}_{4}$ & $\mathfrak{e}_{9}$ & $\mathfrak{e}_{10}$ & $\mathfrak{e}_{10}$ & 
$\mathfrak{e}_{11}$ & $\mathfrak{e}_{15}$ & $\mathfrak{e}_{4}$ & $\mathfrak{e}_{4}$ & $\mathfrak{e}_{9}$ \notag\\
\end{tabular}
\end{center}
\caption{The cells $o_r,e_r,o_r^\ast,e_r^\ast, o_b$ and $e_b$ of $f$ and the cells $e'_r,e^{\ast\prime}_r$ and $e'_b$ of 
$\operatorname{WL}(f)$ in the case when $\mathfrak{d}^{\ast}$ is incident to the bottom right vertex of $o_b$ in $f$.}
\label{Table:Drifter_bottom_right_o_b_exc_2}
\end{table}

In summary, left-Wieland drift has the following effect:
\begin{itemize}
\item The move $M_5$ is applied if $o_r=\mathfrak{o}_9$ and $o_r^\ast=\mathfrak{o}_{10}$.
 \item The move $M_1$/$M_2$ is applied to $\mathfrak{d}^\ast$ before the move $M_2$ is applied to $\mathfrak{d}$ if 
 $o_r=\mathfrak{o}_9$ and $o_r^\ast=\mathfrak{o}_8/\mathfrak{o}_{9}$.
 \item The move $M_1$ is applied to $\mathfrak{d}$ before the move $M_1$/$M_2$/$M_3$ is applied to $\mathfrak{d}^\ast$ if $o_r=\mathfrak{o}_8$ and 
 $o_r^\ast=\mathfrak{o}_8/\mathfrak{o}_9/\mathfrak{o}_{10}$.
 \item The move $M_3$ is applied to $\mathfrak{d}$ before the move $M_1$/$M_2$/$M_3$  is applied to $\mathfrak{d}^\ast$ if 
 $o_r=\mathfrak{o}_{10}$ and $o_r^\ast=\mathfrak{o}_8/\mathfrak{o}_9/\mathfrak{o}_{10}$.
\end{itemize}
In all cases the rest of $f$ is preserved by left-Wieland drift.\\ 

Next, the case when the drifter $\mathfrak{d}^\ast$ is contained in $e_r$ is studied. In that case the $x$-coordinate of $\mathfrak{d}^\ast$ is larger than the one of $\mathfrak{d}$. 
Note that $(o_r,e_r)\in\{(\mathfrak{o}_9,\mathfrak{e}_{10}),(\mathfrak{o}_{10},\mathfrak{e}_9)\}$ 
since $f$ contains neither of the cells $\mathfrak{o}_{11}$ and $\mathfrak{e}_{11}$.
Now, if $o_r=\mathfrak{o}_9$ then $o_b=\mathfrak{o}_7$, $e_b=\mathfrak{e}_4$, $o_t^\ast=\mathfrak{o}_4$, $e_t^\ast=\mathfrak{e}_6$ and 
$(o_r^\ast,e_r^\ast)\in\{(\mathfrak{o}_8,\mathfrak{e}_5),(\mathfrak{o}_9,\mathfrak{e}_2)\}$. 
Furthermore, if $o_r^\ast=\mathfrak{o}_9$ then $o_b^\ast=\mathfrak{o}_7$ and $e_b^\ast=\mathfrak{e}_4$. Thus,
$e'_r=\mathfrak{e}_6$, $e'_b=\mathfrak{e}_{10}$, $e^{\ast\prime}_t=\mathfrak{e}_4$ 
and if $o_r^\ast=\mathfrak{o}_9$ then $e^{\ast\prime}_r=\mathfrak{e}_6$ and $e^{\ast\prime}_b=\mathfrak{e}_{10}$.
On the other hand, if $o_r=\mathfrak{o}_{10}$ then $o_t=\mathfrak{o}_6$, $e_t=\mathfrak{e}_4$, $o_b^\ast=\mathfrak{o}_4$, $e_b^\ast=\mathfrak{e}_7$ and 
$(o_r^\ast,e_r^\ast)\in\{(\mathfrak{o}_8,\mathfrak{e}_5),(\mathfrak{o}_{10},\mathfrak{e}_3)\}$. Furthermore, if $o_r^\ast=\mathfrak{o}_{10}$ then 
$o_t^\ast=\mathfrak{o}_6$ and $e_t^\ast=\mathfrak{e}_4$. Thus, 
$e'_r=\mathfrak{e}_7$, $e'_t=\mathfrak{e}_9$ and 
$e^{\ast\prime}_b=\mathfrak{e}_4$ and if $o_r^\ast=\mathfrak{o}_{10}$ then $e^{\ast\prime}_r=\mathfrak{e}_7$ and $e^{\ast\prime}_t=\mathfrak{e}_9$.
By Lemma~\ref{Lemma:Effected_cells_exc_2} and Lemma~\ref{Lemma:Wieland_boundary_cells_exc_2} the effect of left-Wieland drift is the following:
\begin{itemize}
 \item The move $M_1$/$M_2$ is applied to $\mathfrak{d}^\ast$ before the move $M_2$ is applied to $\mathfrak{d}$ 
 if $o_r=\mathfrak{o}_9$ and $o_r^\ast=\mathfrak{o}_8/\mathfrak{o}_{9}$.
 \item The move $M_1$/$M_3$ is applied to $\mathfrak{d}^\ast$ before the move $M_3$ is applied to $\mathfrak{d}$ 
 if $o_r=\mathfrak{o}_{10}$ and $o_r^\ast=\mathfrak{o}_8/\mathfrak{o}_{10}$.
\end{itemize}
In both cases the rest of $f$ is preserved by left-Wieland drift.\\

Next, the case when $\mathfrak{d}^\ast$ is contained in $e_b$ is regarded. In that case the $x$-coordinate of $\mathfrak{d}^\ast$ is larger than the one of $\mathfrak{d}$. 
The cells $o_l$ and $o_b$ are both not contained in $\{\mathfrak{o}_{6},\mathfrak{o}_{7},\mathfrak{o}_{12}\}$. For instance, it is not possible that $o_b$ equals $\mathfrak{o}_{7}$ because then $e_b$ would have to equal 
$e_{15}$ or the bottom right vertex of $o_b$ would be incident to a drifter. As a start, if $o_t$ exists then it cannot be in $\{\mathfrak{o}_{7},\mathfrak{o}_{12}\}$. Furthermore, if $o_t=\mathfrak{o}_6$ then 
$e_t=\mathfrak{e}_4$, $o_r=\mathfrak{o}_{10}$, $e_r=\mathfrak{e}_3$, $e_t'=\mathfrak{e}_9$ and $e_r'=\mathfrak{e}_7$. On the other hand, $o_b^\ast$ cannot be in $\{\mathfrak{o}_{6},\mathfrak{o}_{12}\}$. Furthermore, 
if $o_b^\ast=\mathfrak{o}_7$ then $e_b^\ast=\mathfrak{e}_4$, $o_r^\ast=\mathfrak{o}_{9}$, $e_r^\ast=\mathfrak{e}_2$, $e_b^{\ast\prime}=\mathfrak{e}_{10}$ and $e_r^{\ast\prime}=\mathfrak{e}_6$.
To determine the effect of left-Wieland drift on $f$ it remains to study the cells $o_r^\ast$, $e_r^\ast$, $o_r$, $e_r$, $e_r^{\ast\prime}$ and $e_r'$. In Table~\ref{Table:Drifter_o_b_exc_2} all possible configurations 
in these cells are listed.

\begin{table}[tbh]
\begin{center}
\begin{tabular}{c!{\vrule width 1.35pt}c!{\vrule width 1pt}c!{\vrule width 1pt}c!{\vrule width 1pt}c!{\vrule width 1pt}c}
$o_r^\ast$ & $\mathfrak{o}_{8}$ & $\mathfrak{o}_{8}$ & $\mathfrak{o}_{9}$ & $\mathfrak{o}_{9}$ & $\mathfrak{o}_{10}$ \notag\\ 
\hline
$e_r^\ast$ & $\mathfrak{e}_{5}$ & $\mathfrak{e}_{5}$ & $\mathfrak{e}_{2}$ & $\mathfrak{e}_{2}$ & $\mathfrak{e}_{3}$  \notag\\
\hline
$o_r$ & $\mathfrak{o}_{8}$ & $\mathfrak{o}_{10}$ & $\mathfrak{o}_{8}$ & $\mathfrak{o}_{10}$ & $\mathfrak{o}_{13}$ \notag\\
\hline
$e_r$ & $\mathfrak{e}_{2}$ & $\mathfrak{e}_{1}$ & $\mathfrak{e}_{2}$ & $\mathfrak{e}_{1}$ & $\mathfrak{e}_{4}$  \notag\\
\noalign{\hrule height 1.25pt}
$e_r^{\ast\prime}$ & $\mathfrak{e}_{8}$ & $\mathfrak{e}_{8}$ & $\mathfrak{e}_{6}$ & $\mathfrak{e}_{6}$ & $\mathfrak{e}_{7}$  \notag\\ 
\hline
$e_r'$ & $\mathfrak{e}_{8}$ & $\mathfrak{e}_{7}$ & $\mathfrak{e}_{8}$ & $\mathfrak{e}_{7}$ & $\mathfrak{e}_{14}$ \notag\\
\end{tabular}
\end{center}
\caption{The cells $o_r^\ast$, $e_r^\ast$, $o_r$ and $e_r$ of $f$ and the cells $e_r^{\ast\prime}$ and $e'_r$ of $\operatorname{WL}(f)$ in the case when 
$\mathfrak{d}^{\ast}$ is contained in $e_b$.}
\label{Table:Drifter_o_b_exc_2}
\end{table}

In summary, the effect of left-Wieland drift on $f$ is the following:
\begin{itemize}
 \item The move $M_1$ is applied to $\mathfrak{d}^\ast$ before the move $M_1$/$M_3$ is applied to $\mathfrak{d}$ if $o_r^\ast=\mathfrak{o}_8$ and $o_r=\mathfrak{o}_8/\mathfrak{o}_{10}$.
 \item The move $M_2$ is applied to $\mathfrak{d}^\ast$ before the move $M_1$/$M_3$ is applied to $\mathfrak{d}$ if $o_r^\ast=\mathfrak{o}_9$ and $o_r=\mathfrak{o}_8/\mathfrak{o}_{10}$.
 \item The move $M_3$ is applied to $\mathfrak{d}^\ast$ before it also is applied to $\mathfrak{d}$ if $o_r^\ast=\mathfrak{o}_{10}$ and $o_r=\mathfrak{o}_{13}$.
\end{itemize}
In all cases the rest of $f$ is preserved by left-Wieland drift.\\

The last case that is to be considered is the case when $\mathfrak{d}^\ast$ is contained in $o_b$. In that case the $x$-coordinate of $\mathfrak{d}$ is larger than the one of $\mathfrak{d}^\ast$. 
Furthermore, the cells $o_l$ and $o_l^\ast$ are not contained in $\{\mathfrak{o}_6, \mathfrak{o}_7, \mathfrak{o}_12\}$, if $o_t$ exists then it cannot be in $\{\mathfrak{o}_{7},\mathfrak{o}_{12}\}$ and 
$o_b^\ast$ cannot be in $\{\mathfrak{o}_{6},\mathfrak{o}_{12}\}$. 
On the other hand, if $o_t=\mathfrak{o}_6$ then $e_t=\mathfrak{e}_4$, $o_r=\mathfrak{o}_{10}$, $e_r=\mathfrak{e}_3$, $e_t'=\mathfrak{e}_9$ and $e_r'=\mathfrak{e}_7$ 
and if $o_b^\ast=\mathfrak{o}_7$ then $e_b^\ast=\mathfrak{e}_4$, $o_r^\ast=\mathfrak{o}_{9}$, $e_r^\ast=\mathfrak{e}_2$, $e_b^{\ast\prime}=\mathfrak{e}_{10}$ and $e_r^{\ast\prime}=\mathfrak{e}_6$.
To determine the effect of left-Wieland drift on $f$ it remains to study the cells $o_r^\ast$, $e_r^\ast$, $o_r$, $e_r$, $e_r^{\ast\prime}$ and $e_r'$. In Table~\ref{Table:Drifter_o_b_exc_2} all possible configurations 
in these cells are listed.

\begin{table}[tbh]
\begin{center}
\begin{tabular}{c!{\vrule width 1.35pt}c!{\vrule width 1pt}c!{\vrule width 1pt}c!{\vrule width 1pt}c!{\vrule width 1pt}c}
$o_r$ & $\mathfrak{o}_{8}$ & $\mathfrak{o}_{8}$ & $\mathfrak{o}_{9}$ & $\mathfrak{o}_{10}$ & $\mathfrak{o}_{10}$ \notag\\ 
\hline
$e_r$ & $\mathfrak{e}_{5}$ & $\mathfrak{e}_{5}$ & $\mathfrak{e}_{2}$ & $\mathfrak{e}_{3}$ & $\mathfrak{e}_{3}$  \notag\\
\hline
$o_b$ & $\mathfrak{o}_{8}$ & $\mathfrak{o}_{9}$ & $\mathfrak{o}_{14}$ & $\mathfrak{o}_{8}$ & $\mathfrak{o}_{9}$ \notag\\
\hline
$e_b$ & $\mathfrak{e}_{3}$ & $\mathfrak{e}_{1}$ & $\mathfrak{e}_{4}$ & $\mathfrak{e}_{3}$ & $\mathfrak{e}_{1}$  \notag\\
\noalign{\hrule height 1.25pt}
$e'_r$ & $\mathfrak{e}_{8}$ & $\mathfrak{e}_{8}$ & $\mathfrak{e}_{6}$ & $\mathfrak{e}_{7}$ & $\mathfrak{e}_{7}$  \notag\\ 
\hline
$e'_b$ & $\mathfrak{e}_{8}$ & $\mathfrak{e}_{6}$ & $\mathfrak{e}_{3}$ & $\mathfrak{e}_{8}$ & $\mathfrak{e}_{6}$ \notag\\
\end{tabular}
\end{center}
\caption{The cells $o_r$, $e_r$, $o_b$ and $e_b$ of $f$ and the cells $e'_r$ and $e'_b$ of $\operatorname{WL}(f)$ in the case when 
$\mathfrak{d}^{\ast}$ is contained in $o_b$.}
\label{Table:Drifter_o_b_exc_2}
\end{table}

By Lemma~\ref{Lemma:Effected_cells_exc_2} and Lemma~\ref{Lemma:Wieland_boundary_cells_exc_2} the effect of left-Wieland drift on $f$ is the following:
\begin{itemize}
 \item The move $M_1$ is applied to $\mathfrak{d}$ before the move $M_1$/$M_2$ is applied to $\mathfrak{d}^\ast$ if $o_r=\mathfrak{o}_{8}$ and $o_b=\mathfrak{o}_8/\mathfrak{o}_9$.
 \item The move $M_2$ is first applied to $\mathfrak{d}$ and then to $\mathfrak{d}^\ast$ if $o_r=\mathfrak{o}_9$ and $o_b=\mathfrak{o}_{14}$.
 \item The move $M_3$ is applied to $\mathfrak{d}$ before the move $M_1$/$M_2$ is applied to $\mathfrak{d}^\ast$ if $o_r=\mathfrak{o}_{10}$ and $o_b=\mathfrak{o}_8/\mathfrak{o}_9$.
\end{itemize}
\end{proof}

The description of the effect of right-Wieland drift on an instable TFPL of excess at most $2$ follows from the description of the effect of left-Wieland drift on an instable TFPL of excess at most $2$ by vertical symmetry.

\begin{Prop}\label{Prop:Moves_Exc2_unoriented_right} 
Let $f$ be an instable TFPL with boundary $(u,v;w)$ such that $\operatorname{exc}(u,v;w)\leq 2$. Furthermore, let $v^-$ be a word so that $v^-\stackrel{\mathrm{v}}{\longrightarrow}v$. Then the image of $f$ under right-Wieland drift with respect to $v^-$ is determined as follows:
\begin{enumerate}
\item if $L_i$ in $\mathcal{L}^N$ is incident to a drifter delete that drifter and add a horizontal edge incident to $L_{i-1}$ for $i=2,3,\dots,N$; denote the so-obtained TFPL by $f'$;
\item consider the columns of vertices of $G^N$ that contain a vertex, which is incident to a drifter in $f'$: let $\mathcal{I}=\{\ 2\leq i\leq 2N: \,\, a \,\, vertex\,\, of \,\, the \,\, i-th\,\, column\,\ is\,\, incident\,\, to \,\, a \,\, drifter\,\, in\,\, f'\}$, where the columns of $G^N$ 
are counted from left to right. 
\begin{enumerate}
\item If $\mathcal{I}=\{i<j\}$ apply a move in $\{M_1^{-1},M_2^{-1},M_3^{-1}\}$ to the drifter incident to vertices of the $i$-th column and thereafter apply a move in $\{M_1^{-1},M_2^{-1},M_3^{-1}\}$ to the drifter incident to vertices of the $j$-th column;
\item If $\mathcal{I}=\{i\}$ perform a move in $\{M_4^{-1},M_5^{-1}\}$ or if this is not possible apply a move in $\{M_1^{-1},M_2^{-1},M_3^{-1}\}$ to each of the drifters in $f'$ in the following order (if there are two drifters in $f'$): 
if the even cell that contains the lower drifter is not of the form $\mathfrak{o}_{10}$ (see Figure~\ref{Fig:Cells_TFPL_Exc_2}) move the lower drifter first. Otherwise, move the upper drifter first.
\end{enumerate}
\item run through the occurrences of zero in $v^-$: let $\{i_1<i_2<\cdots<i_{N_0}\}=\{i: v_i^-=0\}$. If $v_{i_j+1}$ is the $j$-th zero in $v$ delete the horizontal edge incident to $R_{i_j+1}$ and add a vertical edge incident to $R_{i_j}$ for $j=1,2,\dots, N_0$.
\end{enumerate}
\end{Prop}

\section{The path of a drifter under Wieland drift for TFPLs of excess at most $2$}\label{Sec:Path_drifter_exc_2}

The focus of this section is on how many iterations of left-Wieland drift (\textit{resp.} right-Wieland drift) are needed to move a drifter in an instable TFPL of excess at most $2$ to the right 
(\textit{resp.} left) boundary. The results of the previous section facilitate the study of the effect of Wieland drift on a drifter 
in an instable TFPL of excess at most $2$. When looking at the moves that describe the effect of Wieland drift on instable TFPLs of excess at most $2$ 
one immediately sees that in the preimage of the move $M_4$ there is one drifter whereas in its image there two and that in the preimage of the move $M_5$ there are two drifters whereas in its image there 
is one. Thus, in order to pursue a drifter one has to decide which drifter to pursue after applying the move $M_4^{-1}$ \textit{resp.} $M_5$.
Hence, fix a drifter in the image of the move $M_5^{-1}$ that is identified with the drifter in the preimage and fix a drifter in the image of the move $M_4$ that 
is identified with the drifter in the preimage. For all other moves, identify the drifter in the image with the drifter in the preimage.\\

Given an instable TFPL $f$ of excess at most $2$ and a drifter $\mathfrak{d}$ in $f$ there exists a unique non-negative integer $L(\mathfrak{d})$ so that 
$\mathfrak{d}$ is incident to a vertex in $\mathcal{L}^N$ in $\operatorname{WL}^{L(\mathfrak{d})}(f)$ and 
$\mathfrak{d}$ is contained in $\operatorname{WL}^{\ell}(f)$ for each $0\leq \ell\leq L(\mathfrak{d})$ by Proposition~\ref{Prop:Moves_Exc2_unoriented}. On the other hand, there exists a unique non-negative integer 
$R(\mathfrak{d})$ so that $\mathfrak{d}$ is incident to a vertex in $\mathcal{R}^N$ in $\operatorname{WR}^{R(\mathfrak{d})}(f)$ and $\mathfrak{d}$ is contained in $\operatorname{WR}^r(f)$ 
for each $0\leq r\leq R(\mathfrak{d})$ by Proposition~\ref{Prop:Moves_Exc2_unoriented_right}. The main result of this section will be 
an expression for the sum $L(\mathfrak{d})+R(\mathfrak{d})$ in terms of boundary words and will be stated in Corollary~\ref{Cor:Path_drifter_exc_2}.

\begin{Def}[$\operatorname{Path}(\mathfrak{d}),\operatorname{Left}(\mathfrak{d}),\operatorname{Right}(\mathfrak{d}),\operatorname{HeightL}(\mathfrak{d}),\operatorname{HeightR}(\mathfrak{d})$]
Let $f$ be an instable TFPL of excess at most $2$ and $\mathfrak{d}$ be a drifter in $f$. The \textnormal{path} of $\mathfrak{d}$ -- denoted by $\operatorname{Path}(\mathfrak{d})$ -- is defined as 
the sequence of all instable TFPLs that contain $\mathfrak{d}$ and can be reached by an iterated application of left- or right-Wieland drift to $f$. That is 
 \begin{equation}
  \operatorname{Path}(\mathfrak{d})=\left(\operatorname{WR}^{R(\mathfrak{d})}(f),\dots,\operatorname{WR}(f),f,\operatorname{WL}(f),\dots,\operatorname{WL}^{L(\mathfrak{d})}(f)\right).
 \notag
 \end{equation}
Furthermore, $\operatorname{WR}^{R(\mathfrak{d})}(f)$ is denoted by $\operatorname{Right}(\mathfrak{d})$ and $\operatorname{WL}^{L(\mathfrak{d})}(f)$ is denoted by $\operatorname{Left}(\mathfrak{d})$.
The unique positive integer $h$ with $\mathfrak{d}$ incident to $L_{h+1}$ in $\operatorname{Right}(\mathfrak{d})$ is denoted by $\operatorname{HeigthR}(\mathfrak{d})$ and the unique positive integer $h'$ with 
$\mathfrak{d}$ incident to $R_{N-h'}$ in $\operatorname{Left}(\mathfrak{d})$ is denoted by $\operatorname{HeightL}(\mathfrak{d})$.
\end{Def}

By definition it holds $\vert\operatorname{Path}(\mathfrak{d})\vert= L(\mathfrak{d})+R(\mathfrak{d})+1$.
For $i=1,2,3,4,5$ set 
\begin{align}
\#M_i(\mathfrak{d}) & = \vert\{0\leq r \leq R(\mathfrak{d})-1: \text{ by $\operatorname{WR}$ the move 
$M_i^{-1}$ is applied to $\mathfrak{d}$ in $\operatorname{WR}^r(f)$}\}\vert\notag\\ 
& + \vert\{0\leq \ell\leq L(\mathfrak{d})-1: \text{ by $\operatorname{WL}$ the move $M_i$ is applied to $\mathfrak{d}$ in $\operatorname{WL}^\ell(f)$}\}\vert.
\notag\end{align}
Thus, $\vert \operatorname{Path}(\mathfrak{d})\vert=\#M_1(\mathfrak{d})+\#M_2(\mathfrak{d})+\#M_3(\mathfrak{d})+\#M_4(\mathfrak{d})+\#M_5(\mathfrak{d})+1$ and in summary 

\begin{equation}
L(\mathfrak{d})+R(\mathfrak{d})=\#M_1(\mathfrak{d})+\#M_2(\mathfrak{d})+\#M_3(\mathfrak{d})+\#M_4(\mathfrak{d})+\#M_5(\mathfrak{d}).
\notag\end{equation}

\begin{Def}[$R_i(u^{R(\mathfrak{d})})$, $L_i(v^{L(\mathfrak{d})})$] Let $f$ be an instable TFPL with boundary $(u,v;w)$ of excess at most $2$ and $\mathfrak{d}$ a drifter in $f$.
When $u^{R(\mathfrak{d})}$ denotes the left boundary of $\operatorname{WR}^{R(\mathfrak{d})}(f)$ and 
$v^{L(\mathfrak{d})}$ denotes the right boundary of $\operatorname{WL}^{L(\mathfrak{d})}(f)$ then define 
$R_i(u^{R(\mathfrak{d})})$ as the number of occurrences of $i$ among the last $(N-1-\operatorname{HeightR}(\mathfrak{d}))$ letters of $u^{R(\mathfrak{d})}$ 
and $L_i(v^{L(\mathfrak{d})})$ as the number of occurrences of $i$ among the first $(N-1-\operatorname{HeightL}(\mathfrak{d}))$ letters of $v^{L(\mathfrak{d})}$
for $i=0,1$.
\end{Def}

\begin{Prop}\label{Prop:Path_drifter_exc_2}
Let $f$ be an instable TFPL with boundary $(u,v;w)$ where $exc(u,v;w)\leq 2$, $\mathfrak{d}$ a drifter in $f$ and the notations as above. Then
\begin{align}\label{Eq:Path_drifter_exc_2}
\#M_1(\mathfrak{d})+\#M_2(\mathfrak{d})+\#M_3(\mathfrak{d})+\#M_4(\mathfrak{d})+\#M_5(\mathfrak{d})=R_1(u^{R(\mathfrak{d})})+L_0(v^{L(\mathfrak{d})})+1
.\notag\end{align}    
\end{Prop}

The proof of Proposition~\ref{Prop:Path_drifter_exc_2} will be the content of the rest of this section. The crucial idea is to regard TFPLs together with their canonical orientation. Before starting with the proof 
a crucial corollary of Proposition~\ref{Prop:Path_drifter_exc_2} is stated.

\begin{Cor}\label{Cor:Path_drifter_exc_2} Let $f$ be an instable TFPL with boundary $(u,v;w)$ where $exc(u,v;w)\leq 2$, $\mathfrak{d}$ a drifter in $f$ and the notations as above. Then 
\begin{equation}
 L(\mathfrak{d})+R(\mathfrak{d})=R_1(u^{R(\mathfrak{d})})+L_0(v^{L(\mathfrak{d})})+1.
\label{Eq:Path_drifter_exc_2}\end{equation}
\end{Cor}

In Figure~\ref{Fig:Path_drifter_example_exc_2} an instable TFPL of excess $2$ and the path of one of its drifters which shall in the following be denoted by $\mathfrak{d}$ are depicted. The drifter $\mathfrak{d}$ 
satisfies $R(\mathfrak{d})=2$, $\operatorname{HeightR}(\mathfrak{d})=1$, $u^{R(\mathfrak{d})}=01101$, $R_1(u^{R(\mathfrak{d})})=2$, $L(\mathfrak{d})=2$, 
$\operatorname{HeightL}(\mathfrak{d})=2$, $v^{L(\mathfrak{d})}=01011$ and $L_0(v^{L(\mathfrak{d})})=1$.

\begin{figure}[tbh]
\begin{center}
 \includegraphics[width=1\textwidth]{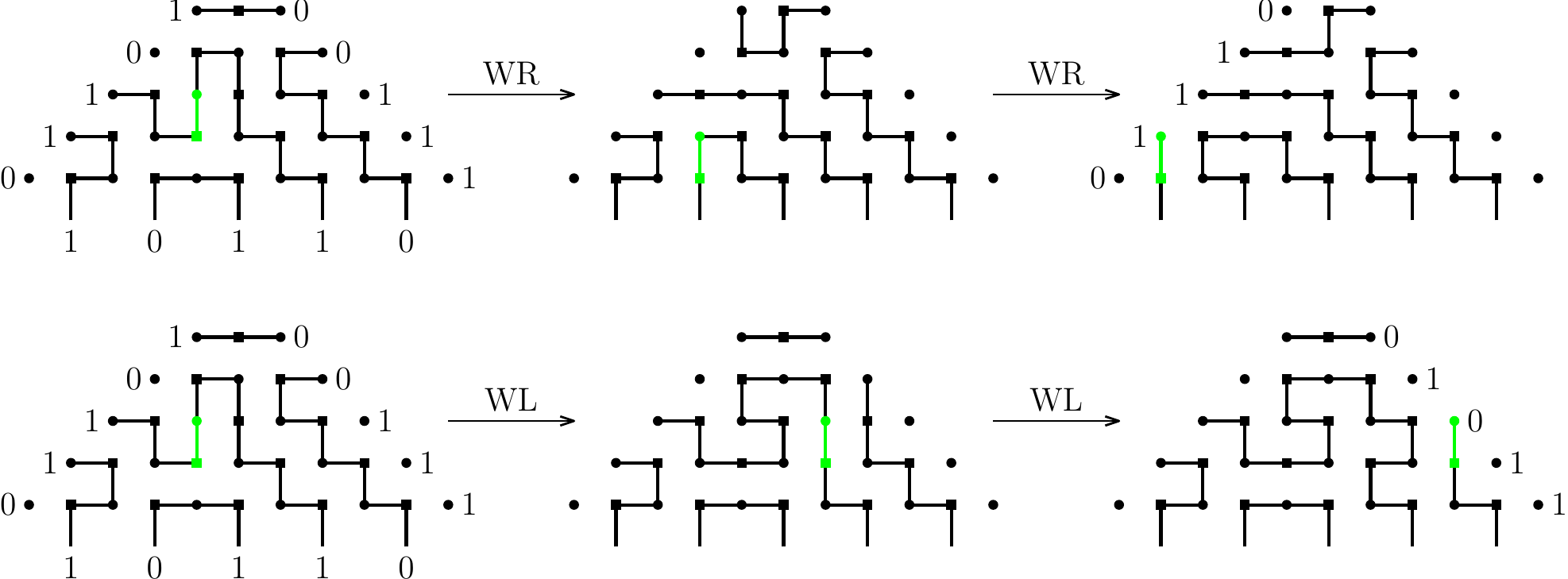}
\caption{A TFPL with boundary $(01101,00111;10110)$ and the path of the drifter that is indicated in green.}
\end{center}
\label{Fig:Path_drifter_example_exc_2}
\end{figure}

The following lemma are immediate consequences of Proposition~\ref{Prop:Comb_Int_Exc}, Proposition~\ref{Prop:Moves_Exc2_unoriented} and Proposition~\ref{Prop:Moves_Exc2_unoriented_right} and describe the effect of 
Wieland drift on canonically oriented TFPLs of excess at most $2$. The moves that form the basis for this description derive from the moves in Figure~\ref{Fig:Moves_Exc2_unoriented} and are depicted in 
Figure~\ref{Fig:Moves_Excess_2_oriented_pathtangle}. In particular, the moves $\overrightarrow{M}_{1,1}$, $\overrightarrow{M}_{1,2}$, $\overrightarrow{M}_{1,3}$, $\overrightarrow{M}_{1,4}$, 
$\overrightarrow{M}_{2,1}$ and $\overrightarrow{M}_{3,1}$ coincide with the moves $BB$, $BR$, $RR$, $RB$, $B$ and $R$ respectively invented in \cite{TFPL} 
for blue-red path tangles corresponding to instable oriented TFPLs of excess $1$.

\begin{figure}[tbh]
\centering
\includegraphics[width=.9\textwidth]{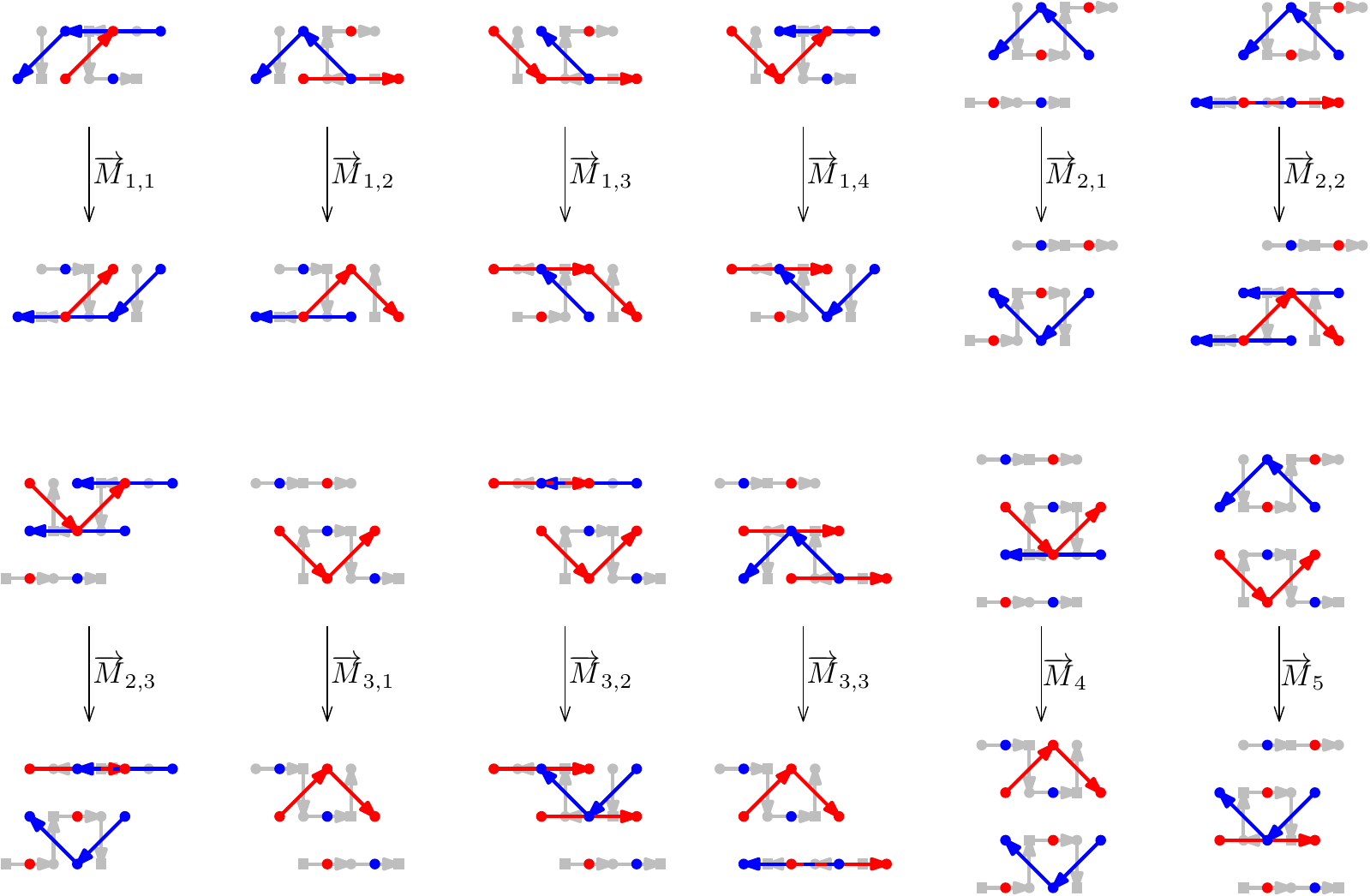}
\caption{The moves describing the effect of Wieland drift on canonically oriented TFPLs of excess at most $2$. In addition, the moves are indicated in terms of blue-red path tangles.}
\label{Fig:Moves_Excess_2_oriented_pathtangle}
\end{figure}

\begin{Lemma}\label{Lem:Wieland_exc_2_oriented} Let $u,v,w$ be words of length $N$ such that $exc(u,v;w)\leq 2$ and $f$ an instable TFPL with boundary $(u,v;w)$ where not all drifters are incident to a vertex in 
$\mathcal{R}^N$. When $\overrightarrow{f}$ denotes $f$ together with the canonical orientation of its edges, then the effect of left-Wieland drift on $f$ translates into the following effect on 
$\overrightarrow{f}$:\begin{enumerate}
      \item If in $\overrightarrow{f}$ there is precisely one drifter 
      then by left-Wieland drift a unique move in\linebreak 
      \hbox{$\{\overrightarrow{M}_{1,1}, \overrightarrow{M}_{1,2}, \overrightarrow{M}_{1,3}, \overrightarrow{M}_{1,4}, 
      \overrightarrow{M}_{2,1}, \overrightarrow{M}_{2,2}, \overrightarrow{M}_{2,3}, \overrightarrow{M}_{3,1}, \overrightarrow{M}_{3,2}, \overrightarrow{M}_{3,3}, \overrightarrow{M}_4\}$} is performed 
      while the rest of $\overrightarrow{f}$ remains unchanged. 
      \item If in $\overrightarrow{f}$ there are two drifters and none of those drifters is incident to a vertex in $\mathcal{R}^N$ 
      then by left-Wieland drift either $\overrightarrow{M}_5$ is performed or to each drifter a unique move in \linebreak
      $\{\overrightarrow{M}_{1,1}, \overrightarrow{M}_{1,2}, \overrightarrow{M}_{1,3}, \overrightarrow{M}_{1,4}, 
      \overrightarrow{M}_{2,1}, \overrightarrow{M}_{3,1}\}$ is applied in the same order as in Proposition~\ref{Prop:Moves_Exc2_unoriented}. The rest of $\overrightarrow{f}$ remains unchanged.
      \item Finally, if in $\overrightarrow{f}$ there are two drifters whereof one is incident to a vertex in $\mathcal{R}^N$ then by left-Wieland drift the drifter incident to a vertex 
      $R_i$ in $\mathcal{R}^N$ is replaced by a horizontal edge 
      incident to $R_{i+1}$ before to the remaining drifter a unique move in $\{\overrightarrow{M}_{1,1}, \overrightarrow{M}_{1,2}, \overrightarrow{M}_{1,3}, \overrightarrow{M}_{1,4}, 
      \overrightarrow{M}_{2,1}, \overrightarrow{M}_{3,1}\}$ is applied. The rest of $\overrightarrow{f}$ remains unchanged by left-Wieland drift.
      \end{enumerate}
 \end{Lemma}
 
 \begin{Lemma}\label{Lem:Wieland_exc_2_oriented_right} Let $u,v,w$ be words of length $N$ such that $exc(u,v;w)\leq 2$ and $f$ an instable TFPL with boundary $(u,v;w)$ where not all drifters are incident to a vertex in 
$\mathcal{L}^N$. Then the effect of right-Wieland drift on $f$ translates into the following effect on $\overrightarrow{f}$:
\begin{enumerate}  
      \item If in $\overrightarrow{f}$ there is precisely one drifter then by right-Wieland drift a unique move in \linebreak
      $\{\overrightarrow{M}^{-1}_{1,1}, \overrightarrow{M}^{-1}_{1,2}, \overrightarrow{M}^{-1}_{1,3}, \overrightarrow{M}^{-1}_{1,4}, \overrightarrow{M}^{-1}_{2,1}, 
      \overrightarrow{M}^{-1}_{2,2}, \overrightarrow{M}^{-1}_{2,3}, \overrightarrow{M}^{-1}_{3,1}, \overrightarrow{M}^{-1}_{3,2}, \overrightarrow{M}^{-1}_{3,3},  
      \overrightarrow{M}^{-1}_5\}$ is performed while the rest of $\overrightarrow{f}$ remains unchanged.
      \item If in $\overrightarrow{f}$ there are two drifters and none of those drifters is incident to a vertex in $\mathcal{L}^N$ then by right-Wieland drift either $\overrightarrow{M}^{-1}_4$ is performed 
      or to each drifter in $\overrightarrow{f}$ a unique move in \linebreak
      $\{\overrightarrow{M}^{-1}_{1,1}, \overrightarrow{M}^{-1}_{1,2}, 
      \overrightarrow{M}^{-1}_{1,3}, \overrightarrow{M}^{-1}_{1,4}, \overrightarrow{M}^{-1}_{2,1}, \overrightarrow{M}^{-1}_{3,1}\}$ 
      is applied in the same order as in Proposition~\ref{Prop:Moves_Exc2_unoriented_right}. The rest of $\overrightarrow{f}$ remains unchanged. 
      \item Finally, if in $\overrightarrow{f}$ there are two drifters whereof one is incident to a vertex in $\mathcal{L}^N$ then by right-Wieland drift 
      the drifter incident to a vertex $L_j$ in $\mathcal{L}^N$ is replaced by a horizontal edge incident to $L_{j-1}$ before a unique move in  
      $\{\overrightarrow{M}^{-1}_{1,1}, \overrightarrow{M}^{-1}_{1,2}, \overrightarrow{M}^{-1}_{1,3}, \overrightarrow{M}^{-1}_{1,4}, \overrightarrow{M}^{-1}_{2,1}, \overrightarrow{M}^{-1}_{3,1}\}$ 
      is applied to the remaining drifter. The rest of $\overrightarrow{f}$ remains unchanged by right-Wieland drift.
      \end{enumerate}      
\end{Lemma}

In the following, for $(i,j)=(1,1),(1,2),(1,3),(1,4),(2,1),(2,2),(2,3),(3,1),(3,2),(3,3)$ set 

\begin{align}
\#\overrightarrow{M}_{i,j}(\mathfrak{d})& = \vert\{0\leq r \leq R(\mathfrak{d})-1: \text{ by $\operatorname{WR}$ the move 
$\overrightarrow{M}_{i,j}^{-1}$ is applied to $\mathfrak{d}$ in $\operatorname{WR}^r(\overrightarrow{f})$}\}\vert\notag\\ 
& + \vert\{0\leq \ell\leq L(\mathfrak{d})-1: \text{ by $\operatorname{WL}$ the move $\overrightarrow{M}_{i,j}$ is applied to $\mathfrak{d}$ in $\operatorname{WL}^\ell(\overrightarrow{f})$}\}\vert.
\notag\end{align}
For the moves $\overrightarrow{M}_4^{-1}$ and $\overrightarrow{M}_5$ it has to be distinguished which of the two drifters in the respective image is identified with the one in the preimage. 
Hence, for $i=4,5$ set 
\begin{align}
\#\overrightarrow{M}_i & = \vert\{0\leq r \leq R(\mathfrak{d})-1: \text{ by $\operatorname{WR}$ the move 
$\overrightarrow{M}_i^{-1}$ is applied to $\mathfrak{d}$ in $\operatorname{WR}^r(\overrightarrow{f})$}\}\vert\notag\\ 
& + \vert\{0\leq \ell\leq L(\mathfrak{d})-1: \text{ by $\operatorname{WL}$ the move $\overrightarrow{M}_i$ is applied to $\mathfrak{d}$ in $\operatorname{WL}^\ell(\overrightarrow{f})$}\}\vert.  
\notag\end{align}
and indicate by a $t$ or $b$ whether in the respective image the \textit{t}op or the \textit{b}ottom drifter is identified with the drifter in the respective preimage. 
In that way, one obtains the notation $\#\overrightarrow{M}_i^b$ respectively $\#\overrightarrow{M}_i^t$ for $i=4,5$.

\begin{Lemma}\label{Lemma:Moves_enumerated_exc_2}
Let $f$ be an instable TFPL with boundary $(u,v;w)$ where $exc(u,v;w)\leq 2$ and $\mathfrak{d}$ a drifter in $f$. When $\overrightarrow{f}$ denotes $f$ together with the canonical orientation of its edges, then the 
following hold:
\begin{enumerate}
 \item $\sum\limits_{i=1}^4\#\overrightarrow{M}_{1,i}(\mathfrak{d})+\sum\limits_{j=1}^3 \#\overrightarrow{M}_{3,j}(\mathfrak{d})+\#\overrightarrow{M}_4^t(\mathfrak{d})+\#\overrightarrow{M}_5^b(\mathfrak{d})=
  N-\operatorname{HeightR}(\mathfrak{d})$,
 \item $\sum\limits_{i=1}^4\#\overrightarrow{M}_{1,i}(\mathfrak{d})+\sum\limits_{j=1}^3 \#\overrightarrow{M}_{2,j}(\mathfrak{d})+\#\overrightarrow{M}_4^b(\mathfrak{d})+\#\overrightarrow{M}_5^t(\mathfrak{d})=
 N-\operatorname{HeightL}(\mathfrak{d})$,
 \item $\#\overrightarrow{M}_{1,1}(\mathfrak{d})+\#\overrightarrow{M}_{1,4}(\mathfrak{d})
 +\#\overrightarrow{M}_{2,3}(\mathfrak{d})-\#\overrightarrow{M}_{3,3}(\mathfrak{d})+\#\overrightarrow{M}_{4}^b(\mathfrak{d})=L_1(v^{L(\mathfrak{d})})$,
 \item $\#\overrightarrow{M}_{1,2}(\mathfrak{d})+\#\overrightarrow{M}_{1,3}(\mathfrak{d})
 -\#\overrightarrow{M}_{2,3}(\mathfrak{d})+\#\overrightarrow{M}_{3,3}(\mathfrak{d})-\#\overrightarrow{M}_{4}^b(\mathfrak{d})=R_0(u^{R(\mathfrak{d})})+1$.
\end{enumerate}
\end{Lemma}

The identities in Lemma~\ref{Lemma:Moves_enumerated_exc_2} generalize the identities of Proposition~6.11 and of Proposition~6.12
in \cite{TFPL} for $\#BB$, $\#BR$, $\#RR$, $\#RB$, $\#B$ and $\#R$. The proof of Lemma~\ref{Lemma:Moves_enumerated_exc_2} is given in terms of blue-red path tangles and uses analogous arguments as the proofs of 
Proposition~6.11 and Proposition~6.12 in \cite{TFPL}. 

\begin{proof} 
As to the first identity, observe that $\overrightarrow{M}_{1,1}$, $\overrightarrow{M}_{1,2}$, $\overrightarrow{M}_{1,3}$, $\overrightarrow{M}_{1,4}$, $\overrightarrow{M}_{3,1}$,
$\overrightarrow{M}_{3,2}$, $\overrightarrow{M}_{3,3}$, $\overrightarrow{M}_4^t$ and $\overrightarrow{M}_5^b$ are the moves that shift the center of the down step in the blue-red path tangle 
corresponding to $\mathfrak{d}$ from one $\backslash$-diagonal of red vertices to the next on the right, 
while the center of the down step corresponding to $\mathfrak{d}$ stays on the same
$\backslash$-diagonal if one of the other moves is performed. Now, the first identity follows since
the center of the blue down step corresponding to $\mathfrak{d}$ in the blue-red path tangle corresponding to $\operatorname{Right}(\mathfrak{d})$ lies on the 
$\operatorname{HeightR}(\mathfrak{d})$-th $\backslash$-diagonal of red vertices when counted from the left whereas the center of the red down step corresponding to $\mathfrak{d}$ in the blue-red path tangle 
corresponding to 
$\operatorname{Left}(\mathfrak{d})$ lies on the $N$-th $\backslash$-diagonal of red vertices.

The second identity follows from the first by symmetry.

For the third identity, note that in the process of moving the down step corresponding to $\mathfrak{d}$ in $\operatorname{Right}(\mathfrak{d})$ to the right boundary by repeatedly applying left-Wieland drift 
it has to ``jump over'' a number of red paths. These are precisely the red paths in the blue-red path tangle corresponding to 
$\operatorname{Left}(\mathfrak{d})$ that end above the red path on which the down step corresponding to $\mathfrak{d}$ lies. Since endpoints of red paths are encoded by ones, the number of these red paths is given by
$L_1(v^{L(\mathfrak{d})})$. Regarding the left-hand side of the third identity, note that a red paths is overcome by the down step corresponding to $\mathfrak{d}$ either by the move 
$\overrightarrow{M}_{1,1}$ or by one of the moves $\overrightarrow{M}_{1,4}$, $\overrightarrow{M}_{2,3}$, $\overrightarrow{M}_{4}^b$ where it is transformed from a red into a blue down step that lies in the area
below the red path.
On the other hand, by the move $\overrightarrow{M}_{3,3}$ the blue down step is transformed into a red down step that lies in the area above the blue path. For that reason, $\#\overrightarrow{M}_{3,3}(\mathfrak{d})$ has 
to be subtracted on the left-hand side of the third identity.

The last identity follows from the third by symmetry.
\end{proof}

\begin{proof}[Proof of Proposition~\ref{Prop:Path_drifter_exc_2}]
By subtracting (4) from (1) in Lemma~\ref{Lemma:Moves_enumerated_exc_2} one obtains
\begin{equation}
R_1(u^{R(\mathfrak{d})})=\#\overrightarrow{M}_{1,1}(\mathfrak{d})+\#\overrightarrow{M}_{1,4}(\mathfrak{d})+\#\overrightarrow{M}_{2,3}(\mathfrak{d})+\#\overrightarrow{M}_{3,1}(\mathfrak{d})
+\#\overrightarrow{M}_{3,2}(\mathfrak{d})+\#\overrightarrow{M}_{4}^b(\mathfrak{d})+\#\overrightarrow{M}_{4}^t(\mathfrak{d})+\#\overrightarrow{M}_{5}^b(\mathfrak{d}).\notag
\end{equation}
On the other hand, by subtracting (3) from (2) in Lemma~\ref{Lemma:Moves_enumerated_exc_2} one obtains
\begin{equation}
L_0(v^{L(\mathfrak{d})})+1=\#\overrightarrow{M}_{1,2}(\mathfrak{d})+\#\overrightarrow{M}_{1,3}(\mathfrak{d})+\#\overrightarrow{M}_{2,1}(\mathfrak{d})+\#\overrightarrow{M}_{2,2}(\mathfrak{d})+
\#\overrightarrow{M}_{3,3}(\mathfrak{d})+\#\overrightarrow{M}_{5}^t(\mathfrak{d}).
\notag\end{equation} 
Summing these two identities gives the assertion.
\end{proof}

\section{Proof of Theorem~\ref{Thm:Expressing_in_stable_TFPLs_exc_2}}\label{Sec:Proof_of_Thm_Expressing_in_stable_TFPLs_exc_2}

In this section, a bijective proof of Theorem~\ref{Thm:Expressing_in_stable_TFPLs_exc_2} is given. By Proposition~\ref{Prop:Path_drifter_exc_2}, precisely one of the inequalities 
\begin{center}
\fbox{$L(\mathfrak{d})\leq L_0(v^{L(\mathfrak{d})})$ and $R(\mathfrak{d})\leq R_1(u^{R(\mathfrak{d})})$} 
\end{center}
is satisfied for each drifter $\mathfrak{d}$ in an instable TFPL $f$ of excess $2$. 
Depending on which of the two inequalities $\mathfrak{d}$ satisfies it is moved to the left or to the right boundary where it then is deleted:
let $f$ be a TFPL with boundary $(u,v;w)$ where $\operatorname{exc}(u,v;w)=2$. A triple $(S(f),g(f),T(f))$ consisting of a semi-standard Young tableau $S(f)$ of skew shape 
$\lambda(u^+)/\lambda(u)$, a stable TFPL $g(f)$ with boundary $(u^+,v^+;w)$ and a semi-standard Young tableau $T(f)$ of skew shape $\lambda(v^+)/\lambda(v)$ is associated with $f$ as follows:
\begin{enumerate}
 \item If $f$ is stable, then set $g(f)=f$, $S(f)$ the empty semi-standard Young tableau of skew shape $\lambda(u)/\lambda(u)$ and $T(f)$ the empty semi-standard Young tableau of skew shape $\lambda(v)'/\lambda(v)'$.
 \item If in $f$ for each drifter $\mathfrak{d}$ it holds $R(\mathfrak{d})\leq R_1(u^{R(\mathfrak{d})})$, then 
 $g(f)=\operatorname{Right}(f)=\operatorname{WR}^{R(f)+1}(f)$, where 
 the boundary of $\operatorname{Right}(f)$ is $(u^+,v;w)$ for a $u^+$ such that $\lambda(u)\subset\lambda(u^+)$, $S(f)$ is the semi-standard Young tableau of skew shape $\lambda(u^+)/\lambda(u)$ corresponding 
 to the sequence
 \begin{equation}
 \lambda(u)=\lambda(u^{0})\subseteq\lambda(u^{1})\subseteq\cdots\subseteq\lambda(u^{R(f)})\subseteq\lambda(u^{R(f)+1})=\lambda(u^+),
 \notag\end{equation}
where $u^{r}$ denotes the left boundary word of $\operatorname{WR}^{r}(f)$ for each $0\leq r\leq R(f)+1$, and $T(f)$ is the empty semi-standard Young tableau of skew shape $\lambda(v)'/\lambda(v)'$.

\item If in $f$ for each drifter $\mathfrak{d}$ it holds $L(\mathfrak{d})\leq L_0(v^{L(\mathfrak{d})})$, 
then set $g(f)=\operatorname{Left}(f)=\operatorname{WL}^{L(f)+1}(f)$, 
where the boundary of $\operatorname{Left}(f)$ is $(u,v^+;w)$ for a $v^+$ such that $\lambda(v)\subset\lambda(v^+)$, $S(f)$ is the empty semi-standard Young tableau of skew shape $\lambda(u)/\lambda(u)$ and $T(f)$ is the 
semi-standard Young tableau of skew shape $\lambda(v^+)'/\lambda(v)'$ corresponding to the sequence
\begin{equation}
\lambda(v)'=\lambda(v^{0})'\subseteq\lambda(v^{1})'\subseteq\cdots\subseteq\lambda(v^{L(f)})'\subseteq\lambda(v^{L(f)+1})'=\lambda(v^+)',
\notag\end{equation}
where $v^{\ell}$ denotes the right boundary word of $\operatorname{WL}^{\ell}(f)$ for each $0\leq \ell\leq L(f)+1$. 

\item If in $f$ there are two drifters $\mathfrak{d}_{\mathfrak{r}}$ and $\mathfrak{d}_{\mathfrak{l}}$ such that $R(\mathfrak{d}_{\mathfrak{r}})\leq R_1(u^{R(\mathfrak{d}_{\mathfrak{r}})})$
and $L(\mathfrak{d}_{\mathfrak{l}})\leq L_0(v^{L(\mathfrak{d}_{\mathfrak{l}})})$, then $g(f)$ is the TFPL with boundary $(u^+,v^+;w)$ for a $u^+$ such that $\lambda(u)\subset\lambda(u^+)$ and a $v^+$ such that 
$\lambda(v)\subset\lambda(v^+)$ 
obtained from $f$ as follows: the drifter $\mathfrak{d}_\mathfrak{r}$ is moved to the left boundary 
using the moves $M_1^{-1}$, $M_2^{-1}$, $M_3^{-1}$ and there replaced by a horizontal edge and the drifter $\mathfrak{d}_\mathfrak{l}$ is moved to the right boundary using the moves $M_1$, $M_2$, $M_3$ 
and there replaced by a horizontal edge. Furthermore, $S(f)$ is the semi-standard Young tableau of skew shape $\lambda(u^+)/\lambda(u)$ with entry $R(\mathfrak{d}_{\mathfrak{r}})+1$
and $T(f)$ is the semi-standard Young tableau of skew shape $\lambda(v^+)'/\lambda(v)'$ with entry $L(\mathfrak{d}_{\mathfrak{l}})+1$. 
\end{enumerate}

In Figure~\ref{Fig:Example_bijection_exc_2}, the TFPL of excess $2$ displayed in Figure~\ref{Fig:Eventually_Stable_Example_exc_2} and the triple $(S,g,T)$ associated with it are depicted.  

\begin{figure}[tbh]
\includegraphics[width=.9\textwidth]{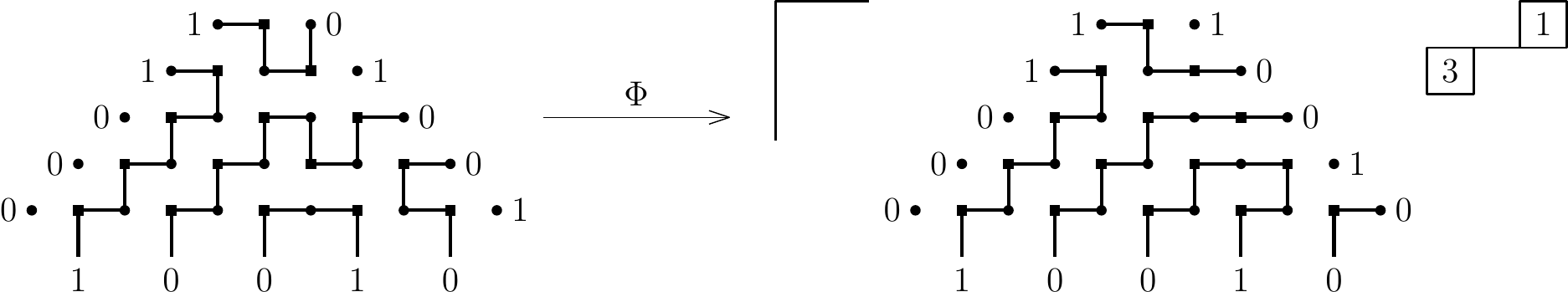}
\caption{An instable TFPL with boundary $(01101,00111;10110)$ and the triple $(S,g,T)$ it is associated with.}
\label{Fig:Example_bijection_exc_2}
\end{figure}
In the following, denote by $G_{\lambda,\lambda^+}$ the set of semi-standard Young tableaux of skew shape $\lambda^+/\lambda$ with entries in the $i$-th column, if counted from right, 
restricted to $1,2,\dots,i$ for Young diagrams $\lambda\subseteq\lambda^+$ which may have empty columns or rows.

\begin{Thm}\label{Thm:Bijection_Exc_2_Stable}
Let $u,v,w$ be words of the same length and with the same number of occurrences of one such that $exc(u,v;w)= 2$. 
Then the map 
\begin{align}
\Phi: T_{u,v}^w & \longrightarrow \bigcup\limits_{u^+, v^+: u^+\geq u, v^+\geq v} G_{\lambda(u),\lambda(u^+)}\times S_{u^+,v^+}^w \times G_{\lambda(v)',\lambda(v^+)'}\notag\\
f & \longmapsto (S(f),g(f),T(f))
\notag\end{align}
is a bijection.
\end{Thm}

\begin{Cor}
The assertion of Theorem~\ref{Thm:Expressing_in_stable_TFPLs_exc_2} immediately follows from Theorem~\ref{Thm:Bijection_Exc_2_Stable}.
\end{Cor}

\begin{Prop}\label{Prop:One_drifter_left_one_right_exc_2}
Let $f$ be an instable TFPL of excess $2$ that contains two drifters $\mathfrak{d}_{\mathfrak{r}}$ and $\mathfrak{d}_{\mathfrak{l}}$ such that $R(\mathfrak{d}_{\mathfrak{r}})\leq R_1(u^{R(\mathfrak{d}_{\mathfrak{r}})})$
and $L(\mathfrak{d}_{\mathfrak{l}})\leq L_0(v^{L(\mathfrak{d}_{\mathfrak{l}})})$. Then $\mathfrak{d}_\mathfrak{r}$ can be moved to the left boundary by the moves $M_1^{-1}$, $M_2^{-1}$ and $M_3^{-1}$ and
$\mathfrak{d}_\mathfrak{l}$ can be moved to the right boundary by the moves $M_1$, $M_2$ and $M_3$.
\end{Prop}

Note that in a TFPL of excess at most $2$ that contains the preimage of the move $M_5$ (\textit{resp.} $M_4^{-1}$) for both drifters $\mathfrak{d}$ and $\mathfrak{d}^\ast$ holds that 
$L(\mathfrak{d})=L(\mathfrak{d}^\ast)$ (\textit{resp.} $R(\mathfrak{d})=R(\mathfrak{d}^\ast)$). Therefore, $\mathfrak{d}$ and $\mathfrak{d}^\ast$ have to satisfy the preconditions of either (1) or (2).
The proof of Proposition~\ref{Prop:One_drifter_left_one_right_exc_2} is based on the following two lemma. 

\begin{Lemma}\label{Lemma:Blockades_exc_2}
Let $f$ be an instable TFPL of excess $2$ that contains two drifters $\mathfrak{d}$ and $\mathfrak{d}^\ast$ whereof at least $\mathfrak{d}$ is not incident to a vertex in $\mathcal{R}^N$ and that does not contain the 
preimage of the move $M_5$. 
If none of the moves $M_1$, $M_2$ or $M_3$ can be applied to $\mathfrak{d}$, then $f$ exhibits one of the following \textnormal{blockades}:
\begin{center}
\includegraphics[width=1\textwidth]{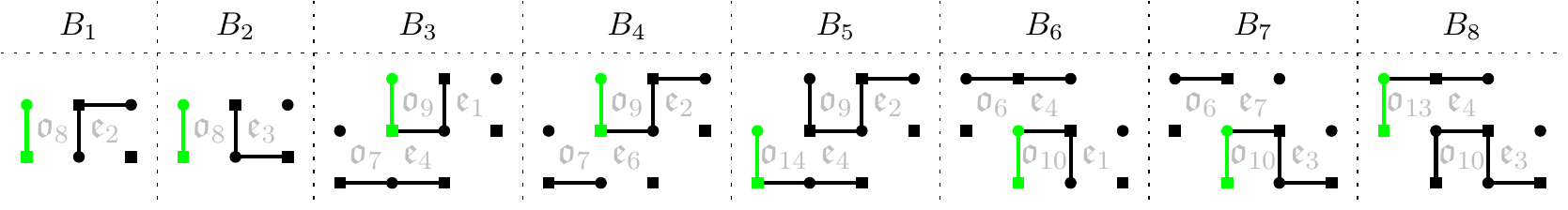} 
\end{center}
\end{Lemma}

\begin{proof}
Since $\mathfrak{d}$ is not incident to a vertex in $\mathcal{R}^N$ it is contained in an odd cell $o$ of $f$. Furthermore, 
$o\in\{\mathfrak{o}_{8},\mathfrak{o}_{9},\mathfrak{o}_{10},\mathfrak{o}_{13},\mathfrak{o}_{14}\}$ because $f$ cannot contain the odd cell $\mathfrak{o}_{11}$ and two drifters at the same time by 
Proposition~\ref{Prop:Comb_Int_Exc}. Here, only the case when 
$o=\mathfrak{o}_8$ is considered. In that case, for the even cell $e$ to the right of $o$ it holds $e\in\{\mathfrak{e}_{2},\mathfrak{e}_{3},\mathfrak{e}_{5}\}$. The case $e=\mathfrak{e}_5$ is impossible since by 
assumption the move $M_1$ cannot be applied to $\mathfrak{d}$. Thus, $e=\mathfrak{e}_2$ or $e=\mathfrak{e}_3$ which give rise to the blockades $B_1$ and $B_2$ respectively. 
\end{proof}

\begin{Lemma}\label{Lemma:Distance_blockades_exc_2}
Let the assumptions be the same as in Lemma~\ref{Lemma:Blockades_exc_2}. Then 
\begin{equation}
L(\mathfrak{d})-L(\mathfrak{d}^\ast)=L_0(v^{L(\mathfrak{d})})-L_0(v^{L(\mathfrak{d}^\ast)})+1. 
\notag\end{equation}
\end{Lemma}

The crucial idea for the proof of Lemma~\ref{Lemma:Distance_blockades_exc_2} is to consider TFPLs of excess at most $2$ together with their canonical orientation and then represent them in terms of blue-red path tangles. 
When doing so the blockades in Lemma~\ref{Lemma:Blockades_exc_2} translate into the blockades depicted in Figure~\ref{Fig:Oriented_blockades_exc_2}.

\begin{figure}[tbh]
\centering
\includegraphics[width=1\textwidth]{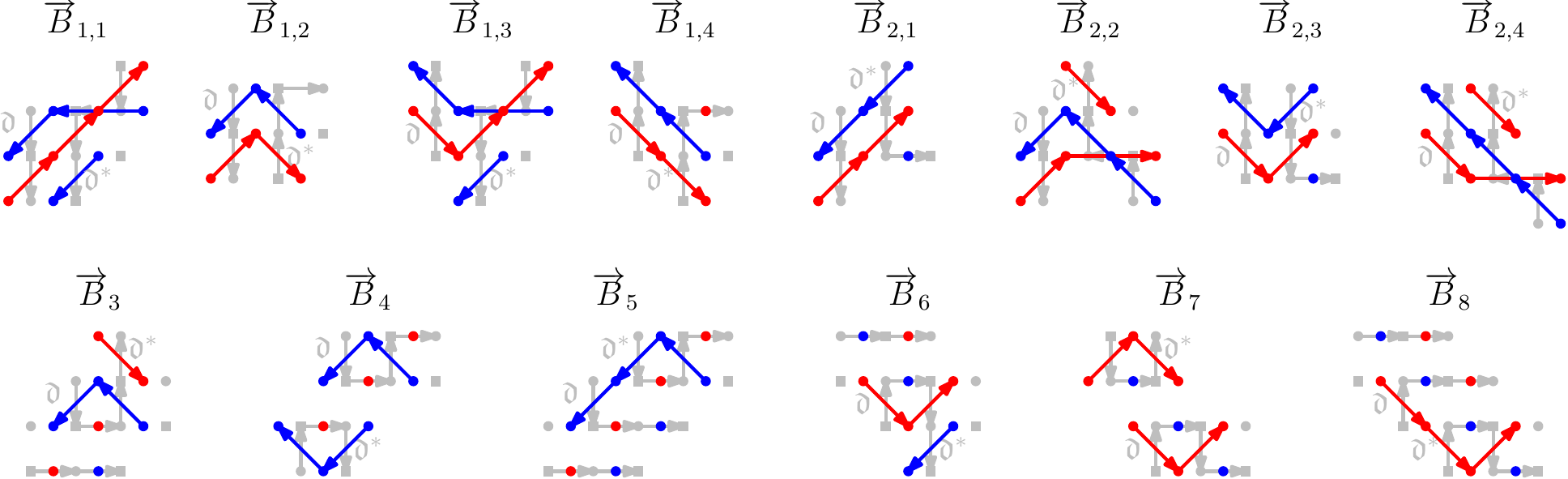}
\caption{The blockades of Lemma~\ref{Lemma:Blockades_exc_2} represented in terms of oriented TFPLs and in terms of blue-red path tangles.}
\label{Fig:Oriented_blockades_exc_2}
\end{figure}

\begin{proof}
In the following, for $(i,j)=(1,1),(1,2),(1,3),(1,4),(2,1),(2,2),(2,3),(3,1),(3,2),(3,3)$ set 
\begin{align}
\#\overrightarrow{ML}_{i,j}(\mathfrak{d}) & = \vert\{ 0\leq \ell\leq  L(\mathfrak{d})-1: \textnormal{ by }  \operatorname{WL} \textnormal{ the move } \overrightarrow{M}_{i,j} \textnormal{ is applied to } \mathfrak{d} \textnormal{ in }
\operatorname{WL}^\ell(\overrightarrow{f}) \}\vert.
\notag\end{align}
The numbers $\#\overrightarrow{ML}_{4}^b(\mathfrak{d})$, $\#\overrightarrow{ML}_{4}^t(\mathfrak{d})$, $\#\overrightarrow{ML}_{5}^b(\mathfrak{d})$ and
$\#\overrightarrow{ML}_{5}^t(\mathfrak{d})$ are defined analogously. Furthermore, set 

\begin{align}
\operatorname{D}_{NW}(\mathfrak{d},\mathfrak{d}^\ast) & = \sum\limits_{i=1}^4\#\overrightarrow{ML}_{1,i}(\mathfrak{d})+\sum\limits_{j=1}^3 \#\overrightarrow{ML}_{3,j}(\mathfrak{d})+
\#\overrightarrow{ML}_4^t(\mathfrak{d})+\#\overrightarrow{ML}_5^b(\mathfrak{d})\notag\\
& - \sum\limits_{i=1}^4\#\overrightarrow{ML}_{1,i}(\mathfrak{d}^\ast)-\sum\limits_{j=1}^3 \#\overrightarrow{ML}_{3,j}(\mathfrak{d}^\ast)-
\#\overrightarrow{ML}_4^t(\mathfrak{d}^\ast)-\#\overrightarrow{ML}_5^b(\mathfrak{d}^\ast),\notag\\
\operatorname{D}_{NE}(\mathfrak{d},\mathfrak{d}^\ast) & = N-\operatorname{HeightL}(\mathfrak{d})-\sum\limits_{i=1}^4\#\overrightarrow{ML}_{1,i}(\mathfrak{d})-
\sum\limits_{j=1}^3 \#\overrightarrow{ML}_{2,j}(\mathfrak{d})
-\#\overrightarrow{ML}_4^b(\mathfrak{d})-\#\overrightarrow{ML}_5^t(\mathfrak{d})\notag\\
& -N+\operatorname{HeightL}(\mathfrak{d}^\ast)+\sum\limits_{i=1}^4\#\overrightarrow{ML}_{1,i}(\mathfrak{d}^\ast)+\sum\limits_{j=1}^3 \#\overrightarrow{ML}_{2,j}(\mathfrak{d}^\ast)+
\#\overrightarrow{ML}_4^b(\mathfrak{d}^\ast)+\#\overrightarrow{ML}_5^t(\mathfrak{d}^\ast),\notag\\
\operatorname{Red}(\mathfrak{d},\mathfrak{d}^\ast) & = L_1(v^{L(\mathfrak{d})})-\#\overrightarrow{ML}_{1,1}(\mathfrak{d})-\#\overrightarrow{ML}_{1,4}(\mathfrak{d})
 -\#\overrightarrow{ML}_{2,3}(\mathfrak{d})+\#\overrightarrow{ML}_{3,3}(\mathfrak{d})-\#\overrightarrow{ML}_{4}^b(\mathfrak{d})\notag\\
& -L_1(v^{L(\mathfrak{d}^\ast)})+\#\overrightarrow{ML}_{1,1}(\mathfrak{d}^\ast)+\#\overrightarrow{ML}_{1,4}(\mathfrak{d}^\ast)
 +\#\overrightarrow{ML}_{2,3}(\mathfrak{d}^\ast)-\#\overrightarrow{ML}_{3,3}(\mathfrak{d}^\ast)+\#\overrightarrow{ML}_{4}^b(\mathfrak{d}^\ast),\notag\\
\operatorname{Blue}(\mathfrak{d},\mathfrak{d}^\ast) & =  \#\overrightarrow{ML}_{1,2}(\mathfrak{d})+\#\overrightarrow{ML}_{1,3}(\mathfrak{d})
 -\#\overrightarrow{ML}_{2,3}(\mathfrak{d})+\#\overrightarrow{ML}_{3,3}(\mathfrak{d})-\#\overrightarrow{ML}_{4}^b(\mathfrak{d})\notag\\
 & - \#\overrightarrow{ML}_{1,2}(\mathfrak{d}^\ast)-\#\overrightarrow{ML}_{1,3}(\mathfrak{d}^\ast)
 + \#\overrightarrow{ML}_{2,3}(\mathfrak{d}^\ast)-\#\overrightarrow{ML}_{3,3}(\mathfrak{d}^\ast)+\#\overrightarrow{ML}_{4}^b(\mathfrak{d}^\ast). 
\notag\end{align}

An easy computation shows that 

\begin{align}
\operatorname{D}_{NW}(\mathfrak{d},\mathfrak{d}^\ast)-\operatorname{Blue}(\mathfrak{d},\mathfrak{d}^\ast)-\operatorname{D}_{NE}(\mathfrak{d},\mathfrak{d}^\ast)
+\operatorname{Red}(\mathfrak{d},\mathfrak{d}^\ast)
=L(\mathfrak{d})-L(\mathfrak{d}^\ast)-L_0(v^{L(\mathfrak{d})})+L_0(v^{L(\mathfrak{d}^\ast)}).
\notag\end{align} 

It remains to show that the lefthand side of the above equation equals $1$. For that purpose, alternative interpretations of the integers 
$\operatorname{D}_{NW}(\mathfrak{d},\mathfrak{d}^\ast)$, $\operatorname{D}_{NE}(\mathfrak{d},\mathfrak{d}^\ast)$, 
$\operatorname{Red}(\mathfrak{d},\mathfrak{d}^\ast)$ and $\operatorname{Blue}(\mathfrak{d},\mathfrak{d}^\ast)$ are considered which derive from Lemma~\ref{Lemma:Moves_enumerated_exc_2}. 

From now on, consider the blue-red path tangle associated with $\operatorname{WL}(\overrightarrow{f})$ for each $0\leq \ell\leq L(\mathfrak{d})$.  
As a start, 
$\operatorname{D}_{NW}(\mathfrak{d},\mathfrak{d}^\ast)$ equals the difference of the number of $\backslash$-diagonals of red vertices to the right of the one 
whereon the center of the down step corresponding to $\mathfrak{d}$ lies and the number of $\backslash$-diagonals of red vertices to the right of the one whereon the center of the down step corresponding to 
$\mathfrak{d}^\ast$ lies by the same arguments as in the proof of Lemma~\ref{Lemma:Moves_enumerated_exc_2}(1).

Furthermore, $\operatorname{D}_{NE}(\mathfrak{d},\mathfrak{d}^\ast)$ equals the number of $/$-diagonals of blue vertices to the left of the one whereon the center of the 
down step corresponding to $\mathfrak{d}$ lies and the number of $/$-diagonals of blue vertices to the left of the one whereon the center of the down step corresponding to $\mathfrak{d}^\ast$ lies by the same 
arguments as in the proof of Lemma~\ref{Lemma:Moves_enumerated_exc_2}.

Next, $\operatorname{Red}(\mathfrak{d},\mathfrak{d}^\ast)$ equals the difference of the number of red paths $\mathfrak{d}$ "jumps over" in the process of moving to the left boundary by repeatedly applying right-Wieland 
drift and the number of red paths $\mathfrak{d}^\ast$ "jumps over" in the process of moving to the left boundary by the iterated application of right-Wieland drift by the same arguments as in the proof of 
Lemma~\ref{Lemma:Moves_enumerated_exc_2}(3). To be more precise, a blue down step ``jumps over'' a red path by the application of $\operatorname{WR}$ if before the application of $\operatorname{WR}$ it is in the area below 
the red path and after the application it is in the area above the red path. On the other hand, a red down step ``jumps over'' a red path if after the application of $\operatorname{WR}$ 
it is a blue down step that lies in the area above the red path. 

Finally, $\operatorname{Blue}(\mathfrak{d},\mathfrak{d}^\ast)$ is the difference of the number of blue paths $\mathfrak{d}$ "jumps over" in the process of moving to the right boundary by repeatedly applying left-Wieland 
drift and the number of blue paths $\mathfrak{d}^\ast$ "jumps over" in the process of moving to the right boundary by the iterated application of left-Wieland drift by the same arguments as in the proof of 
Lemma~\ref{Lemma:Moves_enumerated_exc_2}(4). To be more precise, a red down step ``jumps over'' a blue path by the application of $\operatorname{WL}$ if before the application of $\operatorname{WL}$ it lies in the area below 
the blue path and after the application it is in the area above the blue path. On the other hand, a blue down step ``jumps over'' a blue path if after the application of $\operatorname{WL}$ 
it is a red down step that lies in the area above the red path. 

Thus, for each blockade the integers $\operatorname{D}_{NW}(\mathfrak{d},\mathfrak{d}^\ast)$, $\operatorname{D}_{NE}(\mathfrak{d},\mathfrak{d}^\ast)$, 
$\operatorname{Red}(\mathfrak{d},\mathfrak{d}^\ast)$ and $\operatorname{Blue}(\mathfrak{d},\mathfrak{d}^\ast)$ can be computed separately by looking at Figure~\ref{Fig:Oriented_blockades_exc_2}.
In Table~\ref{Table:Blockades_exc_2}, $\operatorname{D}_{NW}(\mathfrak{d},\mathfrak{d}^\ast)$, $\operatorname{D}_{NE}(\mathfrak{d},\mathfrak{d}^\ast)$, 
$\operatorname{Red}(\mathfrak{d},\mathfrak{d}^\ast)$ and $\operatorname{Blue}(\mathfrak{d},\mathfrak{d}^\ast)$ are listed. 

\begin{table}[tbh]
\begin{center}
\begin{tabular}{c !{\vrule width 1.25pt}c|c|c|c|c|c|c|c|c|c|c|c|c|c}
& $\overrightarrow{B}_{1,1}$ & $\overrightarrow{B}_{1,2}$ & $\overrightarrow{B}_{1,3}$ & $\overrightarrow{B}_{1,4}$ & $\overrightarrow{B}_{2,1}$ &  $\overrightarrow{B}_{2,2}$ &  $\overrightarrow{B}_{2,3}$ &
 $\overrightarrow{B}_{2,4}$ &  $\overrightarrow{B}_{3}$ &  $\overrightarrow{B}_{4}$ & $\overrightarrow{B}_{5}$ & $\overrightarrow{B}_{6}$ & $\overrightarrow{B}_{7}$ & $\overrightarrow{B}_{8}$\\  
\noalign{\hrule height 1.25pt}
$\operatorname{D}_{NW}(\mathfrak{d},\mathfrak{d}^\ast)$ & $0$ & $0$ & $0$ & $0$ &  $1$ &  $1$ & $1$ & $1$ & $1$ & $-1$ & $1$ & $0$ &  $1$ & $0$ \\
\hline
$\operatorname{D}_{NE}(\mathfrak{d},\mathfrak{d}^\ast)$ & $-1$ & $-1$ & $-1$ & $-1$ & $0$ & $0$ & $0$ & $0$ & $0$ & $-1$ & $0$ & $-1$ & $1$ & $-1$ \\
\hline
$\operatorname{Red}(\mathfrak{d},\mathfrak{d}^\ast)$ & $-1$ & $0$ & $-1$ & $0$ & $0$ & $1$ & $0$ & $1$ & $1$ & $0$ & $0$ & $-1$ & $1$ & $0$ \\
\hline
$\operatorname{Blue}(\mathfrak{d},\mathfrak{d}^\ast)$ & $-1$ & $0$ & $-1$ & $0$ & $0$ & $1$ & $0$ & $1$ & $1$ & $-1$ & $0$ & $-1$ & $0$ & $0$ \\
\end{tabular}
\end{center}
\caption{The numbers $\operatorname{D}_{NW}(\mathfrak{d},\mathfrak{d}^\ast)$, $\operatorname{D}_{NE}(\mathfrak{d},\mathfrak{d}^\ast)$,
$\operatorname{Red}(\mathfrak{d},\mathfrak{d}^\ast)$ and $\operatorname{Blue}(\mathfrak{d},\mathfrak{d}^\ast)$
computed separately for each type of blockade.}
\label{Table:Blockades_exc_2}
\end{table}

In summary, it follows that 
\begin{equation}
\operatorname{D}_{NW}(\mathfrak{d},\mathfrak{d}^\ast)-\operatorname{Blue}(\mathfrak{d},\mathfrak{d}^\ast)-\operatorname{D}_{NE}(\mathfrak{d},\mathfrak{d}^\ast)
+\operatorname{Red}(\mathfrak{d},\mathfrak{d}^\ast)=1.
\notag\end{equation}
Therefore, $L(\mathfrak{d})-L(\mathfrak{d}^\ast)-L_0(v^{L(\mathfrak{d})})+L_0(v^{L(\mathfrak{d}^\ast)})=1$.
\end{proof}

Proposition~\ref{Prop:One_drifter_left_one_right_exc_2} follows immediately from Lemma~\ref{Lemma:Distance_blockades_exc_2}.

\begin{proof}[Proof of Proposition~\ref{Prop:One_drifter_left_one_right_exc_2}]
Let $f$ be an instable TFPL of excess $2$ that contains two drifters $\mathfrak{d}_{\mathfrak{l}}$ and $\mathfrak{d}_{\mathfrak{r}}$ such that 
\hbox{$R(\mathfrak{d}_{\mathfrak{r}})\leq R_1(u^{R(\mathfrak{d}_{\mathfrak{r}})})$}
and $L(\mathfrak{d}_{\mathfrak{l}})\leq L_0(v^{L(\mathfrak{d}_{\mathfrak{l}})})$. 
Without loss of generality, suppose that $\mathfrak{d}_\mathfrak{l}$ cannot be moved to the right boundary using the moves $M_1$, $M_2$ and $M_3$. 
Then, by Lemma~\ref{Lemma:Distance_blockades_exc_2}
\begin{equation}
L(\mathfrak{d}_\mathfrak{l})\geq L(\mathfrak{d}_\mathfrak{r})+L_0(v^{L(\mathfrak{d}_\mathfrak{l})})-L_0(v^{L(\mathfrak{d}_\mathfrak{r})})+1.
\notag\end{equation} 
Thus, 
$L_0(v^{L(\mathfrak{d}_\mathfrak{r})})> L(\mathfrak{d}_\mathfrak{r})$ and equivalently $R(\mathfrak{d}_{\mathfrak{r}})> R_1(u^{R(\mathfrak{d}_\mathfrak{r})})+1$. That is a contradiction. 
Therefore, $\mathfrak{d}_\mathfrak{r}$ can be moved to the right boundary using the moves $M_1$, $M_2$ and $M_3$.

By vertical symmetry, $\mathfrak{d}_\mathfrak{l}$ can be moved to the left boundary by the moves $M_1^{-1}$, $M_2^{-1}$ or $M_3^{-1}$.
\end{proof}

In Figure~\ref{Fig:Example_bijection_exc_2_1}, the instable TFPL of excess $2$ of Figure~\ref{Fig:Path_drifter_example_exc_2} and the triple $(S,g,T)$ associated with it are depicted.   

\begin{figure}[tbh]
\includegraphics[width=1\textwidth]{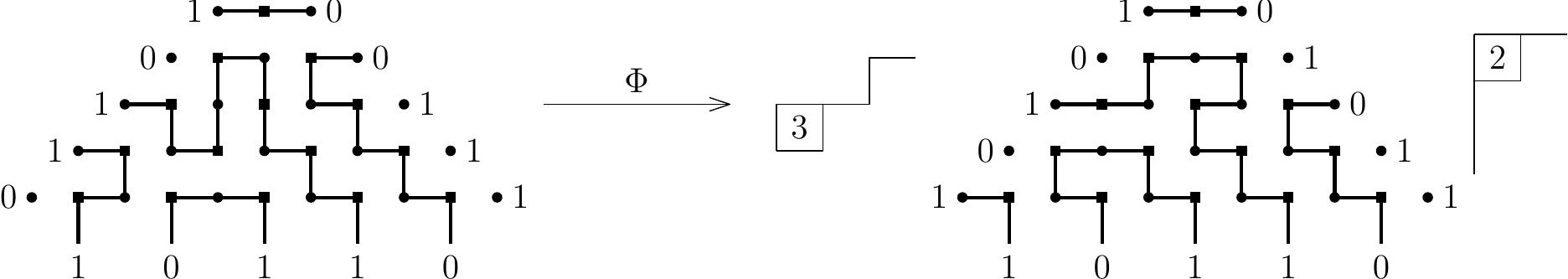}
\caption{The instable TFPL of Figure~\ref{Fig:Path_drifter_example_exc_2} and the triple $(S,g,T)$ it is associated with: $S$ is a semi-standard Young tableau of skew shape $\lambda(10110)/\lambda(01101)$, 
$g$ is a stable TFPL with boundary $(10110,00111;10110)$ and $T$ is the empty semi-standard Young tableau of skew shape $\lambda(00111)'/\lambda(00111)'$.}
\label{Fig:Example_bijection_exc_2_1}
\end{figure}

\begin{proof}[Proof of Theorem~\ref{Thm:Expressing_in_stable_TFPLs_exc_2}] Let $u,v,w$ be words of length $N$ satisfying $exc(u,v;w)\leq 2$.
Furthermore, let $f\in T_{u,v}^w$ be instable and denote by $(S,g,T)$ the image of $f$ under $\Phi$. 
As a start, $S\in G_{\lambda(u),\lambda(u^+)}$ for the following reason: let $c$ be a cell of the 
Young diagram of skew shape $\lambda(u^+)/\lambda(u)$ then its entry in $S$ has to be $R(\mathfrak{d})+1$ for a drifter $\mathfrak{d}$ in $f$ and it has to hold $R(\mathfrak{d})\leq R_1(u^{R(\mathfrak{d})})$ 
by the definition of $\Phi$. Since $R_1(u^{R(\mathfrak{d})})$ is the number of columns to the right of $c$ the entry of $c$ can at most be the number of columns to the right of $c$ plus one. 
By analogous arguments, $T\in G_{\lambda(v)',\lambda(v^+)'}$. 

Thus, it remains to show that $\Phi$ is a bijection. This is done by giving the inverse map: let $u^+\geq u$, $v^+\geq v$, $S\in G_{\lambda(u),\lambda(u^+)}$, $g\in S_{u^+,v^+}^w$ and 
$T\in G_{\lambda(v)',\lambda(v^+)'}$ and consider the sequences 
\begin{equation}
\lambda(u)=\lambda(u^{0})\subseteq \lambda(u^{1})\subseteq \cdots \subseteq  \lambda(u^{R}) \subseteq \lambda(u^{R+1})=\lambda(u^{+})
\notag\end{equation}
corresponding to $S$ so that $R+1$ is the largest entry of $S$ and 
\begin{equation}
\lambda(v)'=\lambda(v^{0})'\subseteq \lambda(v^{1})'\subseteq \cdots \subseteq \lambda(v^{L})'\subseteq \lambda(v^{L+1})'
\notag\end{equation}
corresponding to $T$ so that $L+1$ is the largest entry of $T$. 
Then associate $(S,g,T)$ with a TFPL $\Psi(S,g,T)$ in $T_{u,v}^w$ as follows: 
\begin{enumerate}
\item If $u^+>u$ and $v^+=v$, then set $\Psi(S,g,T)=(\operatorname{WL}_{u^{0}}\circ\operatorname{WL}_{u^{1}}\circ \cdots \circ \operatorname{WL}_{u^{R-1}}\circ\operatorname{WL}_{u^{R}})(g)$.
\item If $u^+=u$ and $v^+>v$, then set $\Psi(S,g,T)=(\operatorname{WR}_{v^{0}}\circ\operatorname{WR}_{v^{1}}\circ \cdots \circ \operatorname{WR}_{v^{L-1}}\circ\operatorname{WR}_{v^{L}})(g)$. 
\item If $u^+>u$ and $v^+>v$, then $\Psi(S,g,T)$ is the TFPL obtained from $g$ as follows: 
since $u^+>u$ and $v^+>v$ the skew shaped Young diagrams $\lambda(u^+)/\lambda(u)$ and $\lambda(v^+)'/\lambda(v)'$ both consist of precisely one cell. Hence, denote by 
$j-1$ the number of columns to the right of the one cell $\lambda(u^+)/\lambda(u)$ contains and by $i_j$ the index of the $j$-th one in $u^+$. On the other hand, denote by 
$j'-1$ the number of columns to the right of the one cell $\lambda(v^+)'/\lambda(v)'$ contains and by $i_j'$ the index of the $j'$-th zero in $v^+$.
Now, a drifter $\mathfrak{d}_\mathfrak{l}$ incident to $L_{i_j+1}$ in $g$ is inserted whereas the horizontal edge incident to $L_{i_j}$ is deleted, 
a drifter $\mathfrak{d}_\mathfrak{r}$ incident to $R_{i_j'-1}$ in $g$ is inserted, whereas the horizontal edge incident to $R_{i_j'}$ is deleted, 
$\mathfrak{d}_\mathfrak{l}$ is moved $L$ times by a move $M_1$, $M_2$ or $M_3$ and $\mathfrak{d}_\mathfrak{r}$ is moved $R$ times by a move $M_1^{-1}$, $M_2^{-1}$ or $M_3^{-1}$. The so-obtained TFPL is the 
image of $(S,g,T)$ under $\Psi$.
\item If $u^+=u$ and $v^+=v$, then $\Psi(S,g,T)=g$.
\end{enumerate}
It can easily be seen that $\Psi$ is the inverse map of $\Phi$.
\end{proof}

\bibliographystyle{plain}

\end{document}